\documentclass[reqno,11pt]{amsart}
\usepackage{amsthm,amsmath,amsfonts,amssymb,amscd,mathrsfs,graphics,diagrams}
\usepackage{txfonts}
\usepackage{hyperref,supertabular}

\DeclareMathOperator{\LG}{LG}
\DeclareMathOperator{\CY}{CY}
\newcommand{\CP}{{\mathbb{CP}}}
\newcommand{\C}{{\mathbb{C}}}

\newcommand{\E}{{\mathscr E}}

\newcommand{\Z}{{\mathbb{Z}}}

\newcommand{\res}{\operatorname{res}}

\newcommand{\pat}{{\partial}}

\newtheorem{thm}{Theorem}[section]
\newtheorem{lm}[thm]{Lemma}

\newtheorem{crl}[thm]{Corollary}

\newtheorem{conj}[thm]{Conjecture}

\theoremstyle{definition}
\newtheorem{rem}[thm]{Remark}

\newtheorem{df}[thm]{Definition}

\theoremstyle{remark}

\title{LG/CY correspondence between $tt^*$ geometries}
\date{\today}
\author{Huijun Fan$^\dag$}
\address{School of Mathematical Sciences, Peking University, Beijing, China}
\email{fanhj@math.pku.edu.cn}
\author{Tian Lan}
\address{School of Mathematical Sciences, Peking University, Beijing, China}
\email{1701110008@pku.edu.cn}
\author{Zongrui Yang}
\address{School of Mathematical Sciences, Peking University, Beijing, China}
\email{1600010619@pku.edu.cn}
\thanks{$^\dag$ Supported by NSFC (11271028, 11325101,11831017,11890660, 11890661).}

\begin{document}
\maketitle
\begin{center}
\emph{Dedicate to Prof.  Armen Glebovich Sergeev for his birthday}
\end{center}

\begin{abstract} The concept of $tt^*$ geometric structure was introduced by physicists (see \cite{CV1, BCOV} and references therein) , and then studied firstly in mathematics by C. Hertling \cite{Het1}. It is believed that the $tt^*$ geometric structure contains the whole genus $0$ information of a two dimensional topological field theory. In this paper, we propose the LG/CY correspondence conjecture for $tt^*$ geometry and obtain the following result. Let $f\in\mathbb{C}[z_0, \dots, z_{n+2}]$ be a nondegenerate homogeneous polynomialof degree $n+2$, then it defines a Calabi-Yau model represented by a Calabi-Yau hypersurface $X_f$ in $\mathbb{CP}^{n+1}$ or a Landau-Ginzburg model represented by a hypersurface singularity $(\C^{n+2}, f)$, both can be written as a $tt^*$ structure. We proved that there exists a $tt^*$ substructure on Landau-Ginzburg side, which should correspond to the $tt^*$ structure from variation of Hodge structures in Calabi-Yau side.  We build the isomorphism of almost all structures in $tt^*$ geometries between these two models except the isomorphism between real structures.  
\end{abstract}

\tableofcontents
\section{Introduction}

In the early of 1980's, two physicists B. Greene and R. Plesser \cite{GP} found a mysterious duality of Hodge numbers between two T-dual Calabi-Yau 3-folds $M$ and $\check{M}$ (called mirror pair):$h^{3-p,q}(M)=h^{p,q}(\check{M})$. In particular, there is  $h^{1,1}(M)=h^{2,1}(\check{M})$. This observation opened the gate to the vast field of mirror symmetry. The equality reflects the equivalence between two quantized theories. One is realized by quantum invariants of symplectic geometry (called A theory or A-model) and the other is realized by quantum invariants of complex geometry in the dual Calabi-Yau manifold (called B theory or B-model). The quantum theory of symplectic geometry is just the Gromov-Witten theory, whose all genus invariants were already defined. The higher genus quantum invariants of the complex geometry have not been build rigorously in mathematics despit of much computation in physical side (see \cite{BCOV, HKQ} and consequent work).   The genus $0$ invariant of the complex geometry of a Calabi-Yau manifold was understood clearly, which is related to the deformation theory of the complex structure, involving the "classical" geometrical structures such as variation of Hodge structure, Gauss-Manin connection, period integration and etc. The proof of genus $0$ mirror symmetry conjecture is the main theorem of the long monograph \cite{HKKPTVVZ} which elucidates the mirror symmetry phenomenon from the mathematical and physical sides. Since then, there are much work concerning the proof of mirror symmetry conjecture between Gromov-Witten theory and the physical BCOV's theory (see \cite{Zi}). Though there are some efforts (see \cite{CL}) to build the higher genus invariant based on the genus $0$ theory, the complete and rigorous mathematical theory 
is still unknown. 

The A theory of a Calabi-Yau (CY) model is the Gromov-Witten theory  correspondes to the Calabi-Yau nonlinear sigma model in physics. The minimizer of the Lagrangian is the holomorphic map from the genus $g$ Riemann surface to the Calabi-Yau manifold. There is another much related model, called Landau-Ginzburg (LG) model, which consists of a noncompact Kahler manifold with a holomorphic function over it, called superpotential. The A theory of LG model was constructed by the first auther, Jarvis and Ruan based on Witten's r-spin theory, called FJRW theory (see \cite{FJR} and consequenct work). Now people realize that the mirror symmetry is not only the dual equivalence between A theory and B theory of Calabi-Yau mirror pair, but forms a sophisticated duality relations between the A and B theories of CY model and LG model. This global dual relation can be briefly described by the following diagram:
\begin{equation}\label{intr-mirr-diag}
\begin{diagram}
\text{LG A theory} &\rTo^{mirror} & \text{LG B theory}\\
\dTo^{\text{\small LG/CY corresp.}}& &\dTo_{\text{\small LG/CY corresp.}}\\
\text{CY A theory} &\rTo^{mirror} & \text{CY B theory}
\end{diagram}
\end{equation}
The above duality is called the global mirror symmetry and the interested reader can find the precise explanation of the above diagram in \cite{CIR}. 

In the following, we take an example to show the duality conjecture in (\ref{intr-mirr-diag}). Let $\C^5$ be the standard complex Euclidean space and $f(z_0,\dots, z_4)=z_0^5+\cdots+z_4^5$ be the quintic polynomial. Then the tuple $(\C^5, f)$ forms a LG model. On the other hand, the zero locus of $f$ in $\CP^4$ defines a Calabi-Yau hypersurface $X_f=\{[z_0,\dots, z_4]\in \CP^4| f(z_0\dots, z_4)=0\}$. This is the Calabi-Yau model. Because of the supersymmetric requirement in physics, we need consider the action of symmetry group of $f$. 

The maximal symmetry group of $f$ is given by 
$$
G_{max}=\left\{(e^{i2\pi \theta_0},\dots, e^{i2\pi \theta_4})|\theta_i=\frac{k}{5}, i,k=0,1,\dots,4\right\}=\Z_5\times \dots \times \Z_5
$$
There is a cyclic subgroup $\langle J\rangle$ of $G_{max}$ generated by the following element 
$$
J=(e^{i2\pi/5 },\dots, e^{i2\pi/5}). 
$$
We say a subgroup $G$ is admissible, if $G$ satisfies $\langle J\rangle <G<G_{max}$. Define $\tilde{G}=G/\langle J\rangle$.

The real LG A model is given by the triple $(\C^5, G, f)$. The mirror LG B model is given by $(\C^5, G^T, f^T=f)$, where $(G^T, f^T)$ is the dual group of $G$ and the dual polynomial of $f$ determined by Berglund-H\"ubsch-Krawitz correspondence (see \cite{Kr}). In the A side, the LG to CY correspondence is given by the correpondence between the LG A theory of  $(\C^5, G, f)$ and the CY A theory of $(X_f/\tilde{G})$. 

The LG/CY correspondence conjecture in A theory says that the FJRW theory of $(\C^5, G, f)$ is equivalent to the GW theory of the orbifold $(X_f/\tilde{G})$. The LG/CY correpondence for this quintic polynomial has been extensively studied in the past ten years. The genus 0 and 1 cases have been proved and higher genus case has been made big progress (see \cite{CR1,CR2,GS}). Another approach to prove the LG/CY correspondence in A theory is using an unified  theory, the so called gauged linear sigma model proposed by Witten \cite{Wi}. Recently, there are much progress along this way (see \cite{FJR2, CFGKS, TX}). There are also many achievements in the LG to LG mirror conjecture (see \cite{HLSW} and reference therein).  When $G=G_{max}$ (in this case, $G^T=\{1\}$), then the LG B theory is given by K. Saito's flat structure theory \cite{S1, S3, S4} (i.e, Frobenius manifold structure determined by a primitive form) plus Givental's quantization theory. However, if $G^T$ is not the trivial group, the LG B theory has not been defined in general. 

The picture of (\ref{intr-mirr-diag}) has also a categoric version in genus $0$, give by the homological mirror symmetry conjecture proposed by Kontsevich and improved later. The interesting thing is that in this version the A model correspondence is unclear but the B model correspondence was already proved by D. Orlov (see \cite{Or}). 

However, all theories appeared in the above discussion, like GW theory, FJRW theory, Saito's flat structure theory, Givental's quantization theory, Orlov's proof of LG/CY correspondence in B side, are holomorphic theories, i. e, all the generating functions appeared in these theory are formal series of holomorphic coordinates or studied in algebraic geometry category.  However, Bershadsky-Cecotti-Ooguri-Vafa found the holomorphic anomaly equation (see \cite{BCOV}), which involves the anti-holomorphic variables of a generating function, and its holomorphic part gives the holomorphic generating function. This implies that in the B side, there should exist a "big theory" containing the holomorphic and the anti-holomorphic information. The so called BCOV theory has been studied and many interesting results have been discovered (see e. g.,  \cite{YY} and \cite{CL}). The genus $0$ part of this "big theory" was considered in \cite{BCOV} and by Strongminger \cite{St}, which is constructed on the universal deformation space of CY 3-fold. This is the prototype of $tt^*$ geometric structure we want to study in this paper. 

 Cecotti and Vafa \cite{CV1} considered the topological and anti-topological fusion of $N=2$ topological field theory and they generalized the special geometry relation to the $tt^*$ geometry structure. Dubrovin studied the integrability of this geometry structure and gave the decription of the correponding isomonodromic deformation \cite{Du}.Later, the $tt^*$ geometric structure was  systematically studied  by C. Hertling \cite{Het1}, who defined the TERP structure and related it to various known structures, e. g.,  variation of twistor structure\cite{Si}, the Cecotti-Vafa(CV) structure, Frobenius type strcture and variation of Hodge structure. In particular,  Combining the CV structure with the Frobenius type structure he defined the CDV structure, which is the Frobenius manifold structure plus an extra real structure compatible with the real tagent bundle. He also used the oscillating integration to construct a TERP structure.

Another approach to build a $tt^*$ geometry on the deformation space of a LG model was developed by the first author \cite{Fan}. By considering the spectrum theory of a twisted Laplacian operator, he can obtain the VHS on the deformation space of the superpotential. By subtle analytical technique, he can obtain the $tt^*$ geometric structure.  We will describe his theory in detail in Section \ref{sec2}. A simplified construction can be found in \cite{T}.

To formulate our result on LG/CY correpondence between $tt^*$ geometry structure, we first give the definition of $tt^*$ geometry (ref. \cite{Het1, I, Fan, T}). 
\begin{df} [$tt^*$ geometry]
A $tt^*$ geometry structure $\E=(H\longrightarrow M,\kappa,\eta,D,C,\bar{C})$ consists of the following data
\begin{itemize}
\item $H\longrightarrow M$ is a complex vector bundle (called the Hodge bundle);
\item a complex anti-linear involution $\kappa: H\longrightarrow H$, i.e. $\kappa^2=Id, \kappa(\lambda \alpha)=\bar{\lambda}\kappa(\alpha), \forall\lambda\in \mathbb{C}$ ($\kappa$ is called the real form);
\item $\eta$ is a nondegenerate pairing on $H$ and together with the real form $\kappa$ induces a Hermitian metric $g(u,v)=\eta(u, \kappa v)$ on $H$ (called $tt^*$ metric);
\item a flat connection $\nabla=D+\bar{D}+C+\bar{C}$ on $H$, where $D+\bar{D}$ is the Chern connection of $g$ (with respect to the holomorphic structure given by $\bar{D}$), $C$ and $\bar{C}$ are $C^{\infty}$(M)-linear maps
\begin{displaymath}
C:C^{\infty}(H)\longrightarrow C^{\infty}(H)\otimes \mathcal{A}^{1,0}(M),\; \bar{C}:C^{\infty}(H)\longrightarrow C^{\infty}(H)\otimes \mathcal{A}^{0,1}(M)
\end{displaymath}
satisfying
\begin{enumerate}
\item $g$ is real with respect to $\kappa$, i.e. $g(\kappa(u),\kappa(v))=\overline{g(u,v)}$.
\item $(D+\bar{D})\kappa=0, \bar{C}=\kappa\circ C\circ\kappa$.
\item $\bar{C}$ is the adjoint of $C$ with respect to $g$, i.e. $g(C_Xu,v)=g(u,\bar{C}_{\bar {X}}v)$.
\end{enumerate}
The operator $C$ is called the Higgs field and the connection $D+C$ is called the Gauss-Manin connection. 
\end{itemize}
\end{df}

Note that the vector bundle $H\longrightarrow M$ is considered as a holomorphic vector bundle, where the holomorphic structure is given by the flat connection $\nabla$. 

\begin{df} Let $\E_i=(H_i\longrightarrow M_i,\kappa_i,\eta_i,D_i,C_i,\bar{C_i}), i=1,2,$ be two $tt^*$ geometric structures. An embedding $\Phi=(\phi, \phi^{\prime})$ of  two holomorphic bundles
\begin{diagram}
H_1&\rTo^{\phi^{\prime}}& H_2\\
\dTo& & \dTo\\
M_1& \rTo^{\phi}& M_2
\end{diagram}
is called an embedding from the $tt^*$ geometric structure $\E_1$ to $\E_2$ if the following hold: $\forall p\in M_1,  X\in T_p M_1, u, v\in (H_1)_p$, 
\begin{enumerate}
\item $\eta_1(u, v)=\eta_2\circ \phi (\phi^{\prime}(u), \phi^{\prime}(v))$.
\item $\kappa_2\circ \phi^{\prime}=\phi^{\prime}\circ \kappa_1$.
\item $\phi^{\prime}((D_1)_X u)=(D_2)_{\phi_*(X)}(\phi^{\prime}(u))$ and $\phi^{\prime}\circ \kappa_1((\bar{D}_1)_{\bar{X}} (\kappa_1(u)))=(\bar{D}_2)_{\overline{\phi_*({X})}}(\kappa_2(\phi^{\prime}(u))$

\item $\phi^{\prime}((C_1)_X u)=(C_2)_{\phi_*(X)}(\phi^{\prime}(u))$
\end{enumerate}
\end{df}

\begin{df}A homogeneous polynomial $f$ defined on $\C^{n+2}$ is called nondegenerate if $0$ is its sole critical point. 
\end{df}
Assume that $f$ is a nondegenerate homogeneous polynomial. The universal deformation of the singularity $(\C^{n+2},f)$ is determined by the Milnor ring 
$$
R_f=\C[z_0,\dots, z_{n+1}]/(\frac{\pat f}{\pat z_0}, \dots, \frac{\pat f}{\pat z_0}).
$$ 
A marginal deformation is a part of the universal deformation given by  
$$
f_u(z)=f(z)+\sum_{j=1}^{k}u_j\phi_j(z),\phi_i \in R_f, \deg\phi_i=\deg f,\quad i=1,\dots, k, 
$$
where $k$ is the dimension of the marginal part in $R_f$. 

On the other hand, we obtain the smooth hypersurface $X_f\subset \CP^{n+1}$ defined by the zero locus of $f$.  About the deformation of complex structures, we have the following result proved in \cite[Theorem 3.21]{Fan}.
\begin{thm}[LG/CY correspondence between moduli numbers]\label{intr-thm-modu-numb} Suppose $f\in\mathbb{C}[z_0, \dots, z_{n+1}]$ is a nondegenerate homogeneous polynomial of degree $d$, then it defines an n-dimensional smooth hypersurface $X_f\subset\mathbb{CP}^{n+1}$. If $n\geq 2$, $d\geq 3$ but except the case $n=2$, $d=4$, then the universal deformation space of $X_f$ is $H^1(X_f, T_{X_f})$, whose dimension 
$m=\dbinom{n+1+d}{d}-(n+2)^2$
equals to the number of marginal deformations in the universal unfolding of the hypersurface singularity $f$. 
\end{thm}

\begin{rem}\label{intr-rem-modu-numb} The case of $n=2$, $d=4$ corresponds to the K3-surface, whose deformation dimension as a complex manifold is $20$, while its algebraic deformation dimension is $19$ which equals the dimension of the marginal deformation of $f$ as a singularity. 
\end{rem}

Assume that the degree $d$ of $f$ equals $n+2$, then $X_f$ is a Calabi-Yau hypersurface. There are also interesting geometric structures on the deformation space $H^1(X_f, T_{X_f})$ of $X_f$. According to the Tian-Todorov lemma, there is no obstruction for the deformation and the deformation space is just $H^1(X_f, T_{X_f})$. However, this space is too small to form a Frobenius manifold structure. By considering the polyvector fields, Barannikov and Kontsevich \cite{BK}   constructed a Frobenius manifold structure on the "extended moduli space of deformations" for any Calabi-Yau manifold $X$, which contains the classical deformation space $H^1(X, T_{X})$. Later this was developed into a deformation theory of DGBV algebra (see \cite{Ma}). However, there is no global $tt^*$ geometry structure over the extended moduli space. 

By Theorem \ref{intr-thm-modu-numb} and Remark \ref{intr-thm-modu-numb}, we denote by $M$ the (algebraic) deformation parametric space of the complex structure of $X_f$ and let $0\in M$ represents the polynomial $f$ or the hypersurface $X_f$ in the parameter space $M$. In Section \ref{sec2}, we will construct the $tt^*$ geometry structure 
$$
\widehat{\E}^{\LG}=(\hat{H}^{\LG}\to M, \hat{\kappa}^{\LG},\hat{\eta}^{\LG}, \hat{D}^{\LG}, \hat{C}^{\LG}, \hat{\bar{C}}^{\LG})
$$ 
of LG model $(\C^{n+2}, f)$, and the $tt^*$ geometry structure 
$$
\E^{\CY}=(H^{\CY}\to M, \kappa^{\CY},\eta^{\CY}, D^{\CY}, C^{\CY}, \bar{C}^{\CY})
$$ 
of CY model $X_f$. Here the fiber $H^{\LG}\to M$ at $u\in M$ is given by $H^{\LG}_{f_u}\cong R_{f_u}$ and the fiber $H^{\CY}\to M$ at $u$ is given by $H^{\CY}_{f_u}\cong H^n_{prim}(X_{f_u})$, where $H^n_{prim}(X_{f_u})$ is the $n$-dimensional primitive cohomology of $X_{f_u}$. 

The paring $\hat{\eta}^{\LG}$ is given by the residue pairing 
\begin{equation}
\hat{\eta}^{\LG}_u(A, B):=\res_{f_u}(A, B), \forall A, B\in R_{f_u}. 
\end{equation}
The paring $\eta^{\CY}$ is given by the Poincare duality:
\begin{equation}
\eta^{\CY}(\alpha, \beta):=\int_{X_f}\alpha\cup\beta, \forall\alpha, \beta\in H_{prim}^n(X_f).
\end{equation}

The Higgs fields $\hat{C}^{\LG}$ and $C^{\CY}$ are defined by the multiplication on the fiber $\hat{H}^{\LG}$ and $H^{\CY}$. The multiplication of the elements on $\hat{H}^{\LG}_{f_u}$ is induced by the ordinary polynomial multiplication in $R_{f_u}$. The multiplication of the elements on $H^{\CY}_{f_u}$ is induced by the cup product of polyvector fields
\begin{equation}
H^p(\wedge^pTX_{f_u})\times H^q(\wedge^qTX_{f_u})\to H^{p+q}(\wedge^{p+q}TX_{f_u})
\end{equation}
and the contraction isomorphism
\begin{equation}
H^p(\wedge^pTX_f)\cong H^{n-p, p}(X_f),  s\mapsto s\dashv \Omega,
\end{equation}
where $\Omega$ is a holomorphic volume form on the Calabi-Yau manifold $X_f$.

For each point $u\in M$, there is a subspace 
$$
H^{\LG}_{f_u}\cong \oplus_{a=0}^\infty R_{f_u}^{(n+2)a}. 
$$
And let $M_{mar}$ be the subspace given by the marginal deforamtion in $M$. To get the LG-CY correspondence, we need to restrict the sub-bundle $H^{\LG}$ to the subspace  $M_{mar}$. \\
From the above discussion, we see that there is a natural choice of data
$$
\E^{\LG}=({H}^{\LG}\to M_{mar}, \kappa^{\LG},\eta^{\LG}, D^{\LG}, C^{\LG}, \bar{C}^{\LG})
$$ 
which we hope to be a $tt^*$ substructure of $\hat{\E}^{LG}$. We shall examine this in {\bf subsection} 2.3.\par
It is natural to state the following conjucture.

\begin{conj}[LG/CY correspondence for $tt^*$ geometries]\label{intr-conj-1} There exists isomorphism 
$$
\Phi: \E^{\LG}\to \E^{\CY}
$$
between two $tt^*$ geometry structure. 
\end{conj}

We can partially prove Conjecture \ref{intr-conj-1}. The main theorem of this paper is as follows:

\begin{thm}\label{intr-main-thm} There is a bundle isomorphism $\Phi: \hat{H}^{\LG}\to H^{\CY}$ such that the following holds:
\begin{enumerate}
\item when restricted to each fiber at $u\in M$, $\Phi$ is the residue map $r$ given by
\begin{align} 
r: H^{LG}_f\cong\oplus_{a=0}^\infty R_f^{(n+2)a}&\mapsto H^n_{prim}(X_f)=H^{CY}_f,\nonumber \\
[A]&\mapsto c_{n,a}(\res\Omega_A)^{n-a, a}, \deg A=(n+2)a, 
\end{align}
where $c_{n,a}$ is a constant depending on $n, a$. 
\item $\Phi^*\eta^{\CY}=\eta^{\LG}$. 
\item $\Phi^*C^{\CY}=C^{\LG}$. 
\end{enumerate}
\end{thm}

The proof of Theorem \ref{intr-main-thm} depends on the pioneer work of Griffiths and Carlson \cite{CG}. The proof will occupy Section \ref{sec3}. 

In Section \ref{sec4}, we will discuss the correspondence of real forms $\kappa^{LG}$ and $\kappa^{CY}$. For quintic polynomials, some examples were studied by physicists,for example,see Aleshkin and Belavin's work in \cite{AB}. However, the general case is unknown. 

\section{tt* geometry structure of LG and CY models}\label{sec2}

\begin{df} 
We say a polynomial W:$\mathbb{C}^n\longrightarrow \mathbb{C}$ is a quasi-homogeneous polynomial with weights $(q_1,...,q_n)\in \mathbb{Q}^n$ ($q_i=a_i/b_i$ and gcd($a_i,b_i$)=1), if for all $\lambda\in\mathbb{C}$, we have
\begin{displaymath}
W(\lambda^{q_1}z_1,...,\lambda^{q_n}z_n)=\lambda W(z_1,...,z_n)
\end{displaymath}
\end{df}

Assume that 0 is the only isolated critical point of $W$ in $\mathbb{C}^n$. This induces that $W$ defines a smooth hypersurface $X_W$ in the weighted projective space $\mathbb{CP}^{n-1}_{(q_1,...,q_n)}$. This hypersurface carries a natural Hodge structure given by the Hodge theory. The most interesting case is that when the weights $(q_1,...,q_n)$ satisfies the Calabi-Yau condition, the hypersurface $X_W$ will be a Calabi-Yau manifold. This is the condition we will consider in the rest of this paper.\par
Consider the famous Milnor fibration $W:\mathbb{C}^n-W^{-1}(0)\longrightarrow\mathbb{C}^*$(\cite{Mi}). It was proved that the fibre $W_t$($t\in \mathbb{C}^*$) has a homotopy type of a union of $\mu$ spheres of dimension $n-1$. Here $\mu$ is the dimension of the Milnor ring $R_W=\mathbb{C}[z_1,...,z_n]/(\partial_1 W,...,\partial_n W)$ as  a complex vector space. The cohomology group $H(W^{-1}(t),\mathbb{C})$ can also be identified with the Milnor ring by taking the Gelfand-Leray forms of some holomorphic $n$-forms (see {\bf subsection 2.3} and {\bf Appendix} C). These cohomology groups $\{H(W^{-1}(t),\mathbb{C})\lvert t\in\mathbb{C}^*\}$ form a flat bundle with the Gauss-Manin connection, which induces a variation of Hodge structures over $\mathbb{C}^*$. The calculation of monodromy will be given in {\bf Appendix} C, which will help us to get the main result in {\bf subsection} 2.3.  \par
With the weights $(q_1,...,q_n)$ we can define a charge on $\{A_idz_1\wedge...\wedge dz_n\}_{i=1}^{\mu}$ when the basis $\{A_i\}$ is given by monomials. The charge $l$ is defined to be 
\begin{displaymath}
l(z^Idz_1\wedge...\wedge dz_n)=\sum_{i=1}^n q_i(I_i+1)
\end{displaymath}
Then here comes the view point of LG-CY correspondence in this paper. It was first conjectured by physicists, when $\hat{c}=\sum_{i=1}^{\mu}(1-2q_i)$ is an integer, there will be an identification between the subset $R_W^{\prime}\in R_W$ (the elements with integer charges) and the Hodge structure on the hypersurface $X_W$ \cite{Ce2}. This is just the B-side of Gepner's idea \cite{Ge} about the LG-CY correspondence.\par

\subsection{$tt^*$ geometry of Landau-Ginzburg model}

Landau-Ginzburg B-model has intimate relations with singularity theory, which studies the germ of singularity of holomorphic functions.  To get the information of the singularity, one must study the deformation theory of the given holomorphic function. Some topological and geometrical information of the singularity appear when we study the geometric sturctures on the deformation parameter space. In the earlier of 80’s, a flat structure on the tangent bundle of deformation space was found by K. Saito and M. Saito \cite{S1} \cite{S2} in the study of the universal deformation of hypersurface singularities, which comes from the existence of `primitive forms'. Later, the same flat structure was found by B. Dubrovin in the study of topological field theory in early 90’s. This structure is called `Frobenius manifold' by B. Dubrovin. A good reference for the Frobenius structure in LG B-model is \cite{ST}.

In 2011, Fan gives a differential geometric approach to the Landau-Ginzburg B-model by a detailed study of twisted Cauchy-Riemann operators and the variation of Hodge structures induced by them. Fan constructed the $tt^*$ geometry associated to a (marginal or relevant) deformation of a nondegenerate quasi-homogeneous polynomial. In particular, when the deformation space is the universal unfolding, Fan finds out that this $tt^*$ structure induces the Frobenius structure constructed by K. Saito and M. Saito using the theory of primitive forms. Here we introduce this $tt^*$ structure. We follow the expositions in Tang \cite{T}.
\par
\begin{df}
Let $f\in\mathbb{C}[z_1, \dots, z_n]$ be a quasi-homogeneous polynomial, it is called non-degenerate if\\
(1) $f$ contains no monomial of the form $z_iz_j$ for $i\neq j$.\\
(2) $f$ has only an isolated singularity at the origin.
\end{df}
The state space of Landau-Ginzburg B-model is the deformation parameter space of the universal unfolding of the singularity $f$, which can be described by the Jacobi ring:
$$R_f:=\mathbb{C}[z_1,...,z_{n}]/I_f,$$ where $I_f=\left \langle \frac{\partial f}{\partial z_1}, \dots, \frac{\partial f}{\partial z_n} \right \rangle$ is the homogeneous ideal generated by $f$. The Milnor ring $R_f$ is finite-dimensional, and we call its dimension $\mu_f=\text{dim }R_f$ the Milnor number. Now we take a monomial basis $\phi_1,...,\phi_{\mu}$ of $R_f$, choose arbitrary $s$ elements $\phi_1,...,\phi_s$ among them and consider the deformation of $f$ given by
$$f(z, u)=f(z)+\sum_{j=1}^{s}u_j\phi_j(z).$$
We call $u_1, \dots, u_s$ the deformation parameters and we think $f(z, u)$ both as a deformation and as a quasi-homogeneous polynomial of $z, u$.
\begin{df}
The deformation parameter $u_j$ is called:\\
(1) relevant, if the weight of $u_j$ is positive;\\
(2) marginal, if the weight of $u_j$ is zero;\\
(3) irrelevant, if the weight of $u_j$ is negative.
\end{df}

Next we study the twisted Cauchy-Riemann operators and their Hodge decompositions. For a holomorphic function $f$ on $\mathbb{C}^n$, define the twisted operators:
$$\bar{\partial}_f=\bar{\partial}+\partial f\wedge,\quad\text{       }\partial_f=\partial+\bar{\partial} \bar{f}\wedge$$
and the twisted Laplacian$$\Delta_f=\bar{\partial}^{\dag}_f\bar{\partial}_f+\bar{\partial}_f\bar{\partial}^{\dag}_f,$$
where $\bar{\partial}^{\dag}_f=-*\partial_{-f}*$ and $\partial^{\dag}_f=-*\bar{\partial}_{-f}*$ are the adjoints of $\bar{\partial}_f$ and $\partial_f$.
Fan proved the following identities imitating classical Hodge theory:
\begin{lm}[\cite{Fan} Chapter 2]We have the following:
$$\partial_f^2=\bar{\partial}_f^2=0, \quad \bar{\partial_f}\partial_f+\partial_f\bar{\partial}_f=0,$$
$$(\bar{\partial}_f^{\dag})^2=(\partial^{\dag}_f)^2=0, \quad\bar{\partial}^{\dag}_f\partial^{\dag}_f+\partial^{\dag}_f\bar{\partial}^{\dag}_f=0,$$
$$[\partial_f,\bar{\partial}^{\dag}_f]=[\bar{\partial}_f,\partial^{\dag}_f]=0,$$
$$[\partial_f,\partial_f^{\dag}]=[\bar{\partial}_f,\bar{\partial}^{\dag}_f]=\Delta_f,$$
and the K$\ddot{a}$hler-Hodge identities:
$$[\partial_f,\Lambda]=-i\bar{\partial}^{\dag}_f, \quad [\bar{\partial}_f,\Lambda]=i\partial^{\dag}_f,$$
$$[\partial^{\dag}_f,L]=-i\bar{\partial}_f, \quad [\bar{\partial}^{\dag}_f,L]=i\partial_f,$$
where L is the Lefschetz operator and $\Lambda$=$*^{-1}\circ$L$\circ*$. We also know that $\Delta_f$ commutes with all the above operators.
\end{lm}
Next we introduce the concept of strongly tame functions and strong deformations, which can guarantee nice spectral behavior of the twisted Laplacian $\Delta_f$, which play crucial roles in the proof of Hodge decomposition for the twisted Cauchy Riemann operator $\partial_f$.
\begin{df}
(1) A holomorphic function $f$ on $\mathbb{C}^n$ is called tame if there is a constant $C>0$ such that $|\nabla f(z)|\geq C$ as 
$|z|\rightarrow\infty$. And if for any $C>0$, we have
$|\nabla f|^2-C|\nabla^2 f|\rightarrow \infty$ as $|z|\rightarrow\infty$, we call $f$ a strongly tame function.\\
(2) Let $F: \mathbb{C}^n\times S\longrightarrow\mathbb{C}$ be a deformation of $f$. It is called a strong deformation of $f$, if the following conditions hold:\\
(i) $\text{sup}_{t\in S}\mu(f_t)<\infty$.\\
(ii) For any $t\in S$, $f_t$ is strongly tame.\\
(iii) For any $t\in S$, $\Delta_{f_t}$ have common domains in the space of $L^2$ forms.
\end{df}
It was proved in \cite{Fan} {\bf Theorem} 2.43 that non-degenerated quasi-homogeneous polynomials with weights $q_i\leq\frac{1}{2}$ are strongly tame, and that marginal and revelant deformations of these polynomials are strong deformations. The concept of strongly tame functions gives the spectrum behavior of the twisted Laplacian as follows:
\begin{thm}[\cite{Fan} {\bf Theorem} 2.40]\textit{If f is a strongly tame holomorphic function, then $\Delta_f$ has purely discrete spectrum, and all the eigenforms form a complete basis of the Hilbert space $L^2(\Lambda^{*}(\mathbb{C}^n))$.}
\end{thm}
Let $\mathcal{H}\subseteq$Dom($\Delta_f$) be the subspace of $\Delta_f$-harmonic forms, $E_\mu$ be the eigenspace with respect to the eigenvalue $\mu$ and $P_\mu$ be the projection from $L^2(\Lambda^{*}(\mathbb{C}^n))$ to $E_{\mu}$. Then we have the spectrum decomposition formulas:
$$L^2(\Lambda^{*}(\mathbb{C}^n))=\mathcal{H}\oplus\oplus_{i=1}^\infty E_{\mu_i},\quad \Delta_f=\sum_i\mu_i P_{\mu_i}.$$
The Green operator $G_f$ of $\Delta_f$ satisfies:
$$G_f\Delta_f+P_0=\Delta_fG_f+P_0=I.$$
This implies the following Hodge-De-Rham decomposition:
\begin{thm}[\cite{Fan} {\bf Theorem} 2.52] There are orthogonal decompositions
$$L^2\Lambda^k=\mathcal{H}^k\oplus im(\bar{\partial}_f)\oplus im(\bar{\partial}^{\dag}_f).$$
In particular, we have the isomorphism
$$H^*_{((2),\bar{\partial}_f)}\cong \mathcal{H}^*,$$
where $H^*_{((2),\bar{\partial}_f)}$ is the cohomology of the following $L^2$-complex:
$$\cdot\cdot\cdot\rightarrow L^2\Lambda^{k-1}\xrightarrow{\bar{\partial}_f}L^2\Lambda^k\xrightarrow{\bar{\partial}_f}L^2\Lambda^{k+1}\rightarrow\cdot\cdot\cdot$$
\end{thm}
The following structure theorem of Hodge groups can be proved by specral sequences:
\begin{thm}[\cite{Fan} {\bf Theorem} 2.66] Assume f is a non-degenerate quasi-homogeneous polynomial on $\mathbb{C}^n$. Then we have
\begin{displaymath}
{\rm dim}\mathcal{H}_f^k=\left\{ \begin{array}{ll}
0 & \textrm{if k$\textless$ n}\\
\mu_f & \textrm{if k=n}
\end{array} \right.
\end{displaymath}
and there exists an explicit isomorphism: 
$$i:\mathcal{H}_f^n\longrightarrow \Omega^n(\mathbb{C}^n)/df\wedge\Omega^{n-1}(\mathbb{C}^n)\cong R_f.$$
Here $R_f$ is the milnor ring of f.
\end{thm}
\begin{rem}
The above isomorphism can be written as: for each $\phi\in\mathcal{H}^n$, there exists a unique $[g]\in R_f=\mathbb{C}[z_1,...,z_n]/J_f$, such that $\phi=gdz_1\wedge\dots\wedge dz_n+\bar{\partial}_f\gamma$ for some $(n-1)$-form $\gamma$.
\end{rem}
Now we begin to introduce the $tt^*$ geometry structure governing the genus-0 Landau-Ginzburg B-model. Take a nondegernerate quasi-homogeneous polynomial $f(z)\in\mathbb{C}[z_1, \dots, z_n]$ and a deformation 
$$F(z, u)=f(z)+\sum_{i=1}^{s}u_i\phi_i(z)$$ of $f(z)$ with each deformation parameter $u_j, j=1,\dots, s$ marginal or revelant. Denote by $M$ the space of parameters $u_1, \dots, u_{s}$, which is a small neighborhood of the origin in $\mathbb{C}^{s}$. Consider the Hodge bundle $\mathcal{H}^n$ over $M$ with fiber on each $t\in M$ the vector space $\mathcal{H}_{F(z,t)}^n$ of $\Delta_{F(z,t)}$-harmonic forms. Take the real form $\kappa$ to be the complex conjugate on $\mathcal{H}^n$. The $tt^*$ metric on $\mathcal{H}^n$ is induced by the natural metric on the trivial complex Hilbert bundle $\Lambda^*_M(\mathbb{C}^n)=L^2\Lambda^*(\mathbb{C}^n)\times M$ given by $(u,v)(t)=\int_{\mathbb{C}^n}u_t\wedge*\bar{v_t}$, $\forall t\in M$. Let $\Pi: \mathcal{A}^*_M(\mathbb{C}^n)\longrightarrow \Gamma(\mathcal{H}^n)$ be the projection from differential forms to their $\Delta_{f_t}$-harmonic parts. We define the $tt^*$ connections:
$$\hat{D}_i=\Pi\circ\partial_i, \quad\hat{\bar{D}}_{\bar{i}}=\Pi\circ\bar{\partial}_{\bar{i}}, \quad i=1,\cdot\cdot\cdot,s.$$
It can be proved that $\hat{D}+\hat{\bar{D}}$ is the Chern connection of the $tt^*$ metric. The Higgs fields $\hat{C}$ and $\hat{\bar{C}}$ are defined by
$$\hat{C}_i=\Pi\circ\partial_iF,\quad\hat{\bar{C}}_{\bar{i}}=\Pi\circ\overline{\partial_iF}.$$
We define the Gauss-Manin connection to be $D+C$. Then the flatness of $\nabla=D+\bar{D}+C+\bar{C}$ is guaranteed by the so-called $tt^*$ (or Cecotti-Vafa) equations, which will be proved in {\bf Applendix} B. This allows us to define a $tt^*$ stucture. If we impose the condition that dim$(M)=\mu$, or more precisely, when the central charge $\hat{c}=\sum_{i=1}^n(1-2q_i)<1$, $M$ can be the deformation parameter space for a universal unfolding of singularity $f$, and the $tt^*$-geometry captures the genus 0 information of the LG B-model. By the procedures in \cite{Fan}, it can also induce the Frobenius structure constructed by K. Saito and M. Saito using the theory of primitive forms.
\begin{rem}
In the definition of $tt^*$-metric $g$, the integration $(u,v)(t)=\int_{\mathbb{C}^n}u_t\wedge*\bar{v_t}$ makes sense because any $\Delta_f$-harmonic form is $L^2$-integrable. In fact, it is known that they exponentially converge to zero at infinity.
\end{rem}
We have the following theorem from the discussion above.
\begin{thm}[The big $tt^*$ structure on Landau-Ginzburg side]\label{big-structure}
Let $F(z,u)$ be a strong deformation of a non-degenerate quasi-homogeneous polynomial f. Assume that the deformations given by $F$ correspond to the marginal and revelant parts in $R_f$. Then there is a $tt^*$ structure on the harmonic bundle $\Delta_f$-harmonic bundle $\mathcal{H}$ over the deformation space U. We denote it by $\widehat{\E}^{\LG}=(\hat{H}^{\LG}\to M, \hat{\kappa}^{\LG},\hat{\eta}^{\LG}, \hat{D}^{\LG}, \hat{C}^{\LG}, \hat{\bar{C}}^{\LG})$. These datas are given by\\
(1) $M$ is the parameter space of the deformations given by $F$.\\
(2) $\hat{H}^{\LG}$=$\mathcal{H}^n$, on each point $t\in M$, the fibre is given by the $\Delta_{F(z,t)}$-harmonic n forms.\\
(3) The real structure $\hat{\kappa}^{\LG}$ is given by the complex conjuagate action on differential forms.\\
(4) The pairing $\hat{\eta}^{\LG}$ is defined as
$$\hat{\eta}^{\LG}(u,v)(t)=\int_{\mathbb{C}^n}u_t\wedge*v_t$$
Here $u,v$ are two sections of $\hat{H}^{\LG}$. The $tt^*$ metric is $\hat{g}^{\LG}(u,v)=\hat{\eta}^{\LG}(u,\hat{\kappa}^{\LG}v)$\\
(5)The $tt^*$ connections are defined as\\
$$\hat{D}_i^{\LG}=\Pi\circ\partial_i, \quad\hat{\bar{D}}_{\bar{i}}^{\LG}=\Pi\circ\bar{\partial}_{\bar{i}}, \quad i=1,\cdot\cdot\cdot,k.$$
(6)The Higgs feilds are defined as
$$\hat{C}_i^{\LG}=\Pi\circ\partial_iF,\quad\hat{\bar{C}}_{\bar{i}}^{\LG}=\Pi\circ\overline{\partial_iF}, \quad i=1,\cdot\cdot\cdot,k.$$
\end{thm}

\subsection{A $tt^*$ geometry substructure of Landau-Ginzburg model}
\subsubsection{Motivation}\
In the previous subsection we saw that the deformation of a strongly tame function gives a $tt^*$ geometry structure. From the graph $(1)$ in the introduction, it is natural to ask if we can find a Calabi-Yau mirror?\par
Let $h:\mathbb{C}^{n+2}\longrightarrow \mathbb{C}$ be a homogeneous polynomial which only has an isolated critical point at the origin. To get a Calabi-Yau manifold, we shall assume that deg$h=n+2$.
\par
As we will see later, the dimension of the $tt^*$ structure on the Landau-Ginzburg side is bigger than the  dimension of the $tt^*$ structure on the Calabi-Yau side (which we will construct in 2.3). To find a reasonable LG-CY correspondence, we must restrict the $tt^*$ structure on the Landau-Ginzburg side to a smaller one. The base space of the CY $tt^*$ structure is the space of the complex structure, and it was proven by Fan\cite{Fan} that the dimension of the marginal part in the Milnor ring $R_h$ coincides with the dimension of the space of complex structures on the Calabi-Yau hypersurface $X_h$. So we believe that the LG phase space in LG-CY correspondence should be the marginal deformation part.
\par
As for the vector bundles, on the CY side we prove in 2.3 that there is a natural $tt^*$ structure on the cohomology group of the primitive form. In complex algebraic geometry, Griffith found that the so called residue map induces an isomorphism between the primitive cohomology on $X_h$ and a subring $R_{h}^{(n+2)*}=\oplus_{a=0}^\infty R_{h}^{(n+2)a}$ of the Milnor ring.\par
From the above considerations, we restrict the phase space on LG side to be the marginal part $M_{mar}$ in the deformation space near $h$, and we restrict the holomorphic bundle $H^{LG}$ to be the subbundle of $\mathcal{H}$ corresponds to the harmonic forms given by $R^{(n+2)*}_h$. A harmonic form in $H^{LG}$ has a holomorphic $\bar{\partial_h}$-representation $\sum A_idz_1\wedge...\wedge dz_{n+2}$ with $A_i\in R^{(n+2)*}_h$.\par
\subsubsection{Invariance with respect to $tt*$ connection and Higgs field}

At first we should recall {\bf Definition 1.2}.\par
 Let $\E_i=(H_i\longrightarrow M_i,\kappa_i,\eta_i,D_i,C_i,\bar{C_i}), i=1,2,$ be two $tt^*$ geometric structures. An embedding $\Phi=(\phi, \phi^{\prime})$ of  two holomorphic bundles
\begin{diagram}
H_1&\rTo^{\phi^{\prime}}& H_2\\
\dTo& & \dTo\\
M_1& \rTo^{\phi}& M_2
\end{diagram}
is called an embedding from the $tt^*$ geometric structure $\E_1$ to $\E_2$ if the following hold: $\forall p\in M_1,  X\in T_p M_1, u, v\in (H_1)_p$, 
\begin{enumerate}
\item $\eta_1(u, v)=\eta_2\circ \phi (\phi^{\prime}(u), \phi^{\prime}(v))$.
\item $\kappa_2\circ \phi^{\prime}=\phi^{\prime}\circ \kappa_1$.
\item $\phi^{\prime}((D_1)_X u)=(D_2)_{\phi_*(X)}(\phi^{\prime}(u))$ and $\phi^{\prime}\circ \kappa_1((\bar{D}_1)_{\bar{X}} (\kappa_1(u)))=(\bar{D}_2)_{\overline{\phi_*({X})}}(\kappa_2(\phi^{\prime}(u))$

\item $\phi^{\prime}((C_1)_X u)=(C_2)_{\phi_*(X)}(\phi^{\prime}(u))$
\end{enumerate}
A simple observation is: if $\E_2=(H_2\longrightarrow M_2,\kappa_2,\eta_2,D_2,C_2,\bar{C}_2)$ is a $tt^*$ structure, and $M_1$ is a complex submanifold in $M_2$ with a holomorphic subbundle $H_2$ of $H_1\lvert_{M_1}$ and $(\kappa_1,\eta_1,D_1,C_1,\bar{C}_1)$ such that  $\E_1=(H_1\longrightarrow M_1,\kappa_1,\eta_1,D_1,C_1,\bar{C}_1)$ satisfies (1)-(4) above, then it satisfies the $tt^*$ equation and hence form a $tt^*$ structure. In this case we say that $\E_1$ is a sub-$tt^*$ structure of $\E_2$.\par
Now consider a non-degenerate polynomial $f:\mathbb{C}^{n+2}\longrightarrow\mathbb{C}$ of degree $n+2$. {\bf subsection} 2.1 shows that the strong deformation $F$ gives a $tt^*$ structure $\widehat{\E}^{\LG}=(\hat{H}^{\LG}\to M, \hat{\kappa}^{\LG},\hat{\eta}^{\LG}, \hat{D}^{\LG}, \hat{C}^{\LG}, \hat{\bar{C}}^{\LG})$. We want to show that the data $H^{LG}\longrightarrow M_{mar}$ in 2.2.1 gives a $tt^*$ substucture.\par
To give a $tt^*$ structure on $H^{LG}\longrightarrow M_{mar}$, we must show that the operators of the $tt^*$ structure on $\mathcal{H}$ preserve the sections of $H^{LG}$. With the relation between these operators, we just need to show the sections of  $H^{LG}$ is preserved by the connection $\hat{D}^{LG}$, the higgs field $\hat{C}^{LG}$, and the real structure $\hat{\kappa}^{LG}$.\par
Assume we have proved that $\Gamma(H^{LG})$ is closed under the real structure, then the statements for $\hat{D}^{LG}$ and $\hat{C}^{LG}$ can be easily proved by using the formulas from Appendix B.\par
\begin{thm}
Assume $\Gamma(H^{LG})$ is closed under the real structure. Restrict the $\Delta_F$-harmonic bundle $\mathcal{H}$ to the marginal deformation part (which corresponds to $R_F^{d*}$ in the Milnor ring), then for a marginal tagent field $i$,  the space of sections of $H^{LG}$ is closed under $\hat{D}^{LG}_i$ and $\hat{C}^{LG}_i$.
\end{thm}
\begin{proof}
The statements for $D$ and $C$ can be proved easily by using the formulas we give in Appendix B. \par
A section of $H^{LG}$ $A:M\longrightarrow H^{LG}$ has the form $A(u)=G(u)\Pi(\sum A_jdz_1\wedge...\wedge dz_{n+2})$. Here $u=(u_1,u_2,...)$ is a local coordinate of the deformation space of $f(z)=F(z,0)$, $G(u)$ is a smooth function, $\Pi$ is the projection of a $\bar{\partial_f}$ closed form to its $\Delta_F$-harmonic part, and $\{A_i\}$ are polynomials in $R_f^{(n+2)*}$. \par
From Appendix B we have the following formula
$$\hat{D}_i^{\LG}=\Pi\circ\partial_i, \quad\hat{\bar{D}}_{\bar{i}}^{\LG}=\Pi\circ\bar{\partial}_{\bar{i}}$$
$$\hat{C}_i^{\LG}=\Pi\circ\partial_iF,\quad\hat{\bar{C}}_{\bar{i}}^{\LG}=\Pi\circ\overline{\partial_iF}$$
Here $\partial_i$ means take derivative with respect to the direction $u_i$ and $\partial_if$ means multiply the $\bar{\partial_f}$-closed form with the function $\partial_if$ given by the deformation direction $i$. Similar for $\hat{\bar{D}}_{\bar{i}}^{\LG}$ and $\hat{\bar{C}}_{\bar{i}}^{\LG}$.\par
By using the above formulas, we see that the sections of $H^{LG}$ with the form $A_jdz_1\wedge...\wedge_{n+2}+\bar{\partial}_FR_j$, $degA_j=(n+2)*$ gives a holomorphic basis of $H^{LG}$ with respect to the holomorphic structure given by $\hat{\bar{D}}^{LG}$, i.e.
$$\hat{\bar{D}}^{LG}(A_jdz_1\wedge...\wedge_{n+2}+\bar{\partial}_FR_j)=0$$
As we have $\hat{D}^{LG}=\hat{\kappa}^{LG} \circ \hat{\bar{D}}^{LG} \circ \hat{\kappa}^{LG}$ , we see that $\Gamma(H^{LG})$ is closed under $\hat{D}^{LG}$.\par
As for the Higgs field $C$, notice that in our case $u_i$ corresponds to a marginal deformation direction, so $\partial_if$ can be represented by a marginal polynomial, i.e. a polynomal of degree $d$. So the statement holds for $\hat{C}^{LG}$.
\end{proof}
\subsubsection{Invariance with respect to the real structure}
To see the subbundle $H^{\LG}$ is closed under the real structure, we need to review some results.\par
It was noticed early that the $\bar{\partial}_f$-cohomology can be identified with some relative cohomology groups. In \cite{Fan}, Fan estabulished such an identification.\par
Let $f$ be a non-degenerate quasi-homogeneous polynomial on $\mathbb{C}^{n+2}$.\par
Let $\alpha > 0$, define two sets $f^{\geq \alpha}$ and $f^{\leq -\alpha}$ 
\begin{displaymath}
f^{\geq \alpha}=\{z\in \mathbb{C}^{n+2}\lvert Ref(z)\geq \alpha\}, f^{\leq -\alpha}=\{z\in \mathbb{C}^{n+2}\lvert Ref(z)\leq -\alpha\}
\end{displaymath}
As the polynomial $f$ is non-degenerate, for arbitrary $\alpha,\beta > 0$, $f^{\geq \alpha}$ and $f^{\geq \beta}$ are homotopy to each other in $\mathbb{C}^{n+2}$, same for $f^{\leq -\alpha}$ and $f^{\leq -\beta}$. So we can define these two equivalent classes by $f^{\geq +\infty}$ and $f^{\leq -\infty}$.\par
Now consider the relative homology group $H_*(\mathbb{C}^{n+2},f^{\geq +\infty},\mathbb{Z})$. By using the exact sequence 
\begin{displaymath}
...\longrightarrow H_k(f^{\geq +\infty},\mathbb{Z})\longrightarrow H_k(\mathbb{C}^{n+2},\mathbb{Z})\longrightarrow H_k(\mathbb{C}^{n+2},f^{\geq +\infty},\mathbb{Z})\longrightarrow  H_{k-1}(f^{\geq +\infty},\mathbb{Z})\longrightarrow...
\end{displaymath}
we know that the relative homology group  $H_k(\mathbb{C}^{n+2},f^{\geq +\infty},\mathbb{Z})$ is isomorphic to $H_{k-1}(f^{\geq +\infty},\mathbb{Z})$ by the boundary map. The isomorphism between the cohomology groups $H^k(\mathbb{C}^{n+2},f^{\leq -\infty},\mathbb{Z})$ and $H^{k-1}(f^{\geq +\infty},\mathbb{Z})$ can also be obtained by considering the exact sequence. We also have $H^k(\mathbb{C}^{n+2},f^{\geq +\infty},\mathbb{C})=H^k(\mathbb{C}^{n+2},f^{\geq +\infty},\mathbb{Z})\bigotimes\mathbb{C}$. The definition of $H^k(\mathbb{C}^{n+2},f^{\leq -\infty},\mathbb{C})$ is similar. \par
We have the following theorem from the results given by subsection {\bf 4.3} in \cite{Fan}.\par
\begin{lm}[\cite{Fan} Prop 4.52 ]\label{iso-harmonic-lef}
Given a quasi-homogeneous polynomial $f$ with only isolated singularities at the origin. Let $\mathcal{H}$ be the space of $\Delta_f$-harmonic (n+2)-forms. Then there is an isomorphism $\phi:\mathcal{H}\longrightarrow H^{n+2}(\mathbb{C}^{n+2},f^{\leq -\infty})$, defined by
\begin{displaymath}
\phi(\alpha)=e^{f+\bar{f}}\alpha.
\end{displaymath}
To be more explict, consider $\phi(\alpha)$ as an element in the relative cohomology group and let $\Gamma$ be a Lefshetz thimble, then the action of $\phi(\alpha)$ on $\Gamma$ is given by
$$\int_{\Gamma}e^{f+\bar{f}}\alpha$$
\end{lm}
\begin{rem}
It is easy to see that the real structures on these two spaces are both given by the complex conjugate, and $\phi$ is an isomorphism respect to these two real structures.
\end{rem}
To study the above isomorphism, we should discuss the oscillatory integrals. We need to recall some results from \cite{Fan}\par
Let $\alpha$ be a $\Delta_f$-harmonic form, we have the following isomorphism from {\bf Theorem 2.8}
$$i:\mathcal{H}_f\longrightarrow \Omega^{n+2}(\mathbb{C}^{n+2})/df\wedge \Omega^{n+1}(\mathbb{C}^{n+2})\cong R_f$$
given by 
$$\alpha=i(\alpha)+\bar{\partial}_fR_{\alpha}$$
Here we identify $i(\alpha)$ with $p_{\alpha}dz_1\wedge...\wedge dz_{n+2}$, $p_{\alpha}$ is a polynomial.\par
Here is an important theorem from \cite{Fan}
\begin{lm}[\cite{Fan} Lemma 4.88]\label{osc-int}
Let $\{\alpha_a,a=1,...,\mu\}$ be a frame of the Hodge bundle $\mathcal{H}$. Let $R$ be a smooth n+1 form. Let $\{\Gamma_a^{-},a=1,...,\mu\}$ be a basis of 
$H_{n+2}(\mathbb{C}^{n+2},f^{\leq -\infty},\mathbb{C})$ consists $\mu$ Lefthetz thimbles. If at any point $(\tau,t)\in U$, there is 
$$\Sigma_{a=1}^{\mu}\int_{\mathbb{C}^{n+2}}\lvert R\lvert\cdot\lvert\alpha_a\lvert<\infty$$
then
$$\int_{\Gamma_a^{-}}e^{f+\bar{f}}\bar{\partial}_fR=\int_{\Gamma_a^{-}}e^{f+\bar{f}}\partial_fR=0$$
In particular this is true if $R$ has only polynomial growth.
\end{lm}
Apply this theorem to the oscillatory integrals of $\Delta_f$-harmonic forms, we have
\begin{lm}[\cite{Fan} Prop4.89]
When $f:\mathbb{C}^{n+2}\longrightarrow\mathbb{C}$ is a non-degenerate quasi-homogeneous polynomial, we have 
$$\int_{\Gamma_a^{-}}e^{f+\bar{f}}\alpha=\int_{\Gamma_a^{-}}e^{f+\bar{f}}[i(\alpha)+\bar{\partial}_fR_{\alpha}]=\int_{\Gamma_a^{-}}e^{f+\bar{f}}i(\alpha)$$
Here $\alpha$ is a $\Delta_f$-harmonic form and $i(\alpha)$ is a holomorphic $n$-form with polynomial coefficient.
\end{lm}
\begin{proof}
When $f$ is a non-degenerate quasi-homogeneous polynomial, we can take generators of $\Omega^{n+2}(\mathbb{C}^{n+2})/df\wedge \Omega^{n+1}(\mathbb{C}^{n+2})\cong R_f$ with polynomial growth and the corresponding $R_{\alpha}$ can be also chosen to have only polynomial growth. Then apply {\bf Lemma} \ref{osc-int}.
\end{proof}
To calculate the above oscillatory integrals, it is useful to choose a special basis of $H^{n+2}(\mathbb{C}^{n+2},f^{\leq -\infty},\mathbb{Z})$. From now, for simplicity,{\bf we assume that $f$ is a non-degenerate homogeneous polynomial}, the results below can also been obtained in quasi-homogeneous cases by using the same methods.
\begin{lm}\label{Lef-thim} 
Let $f:\mathbb{C}^{n+2}\longrightarrow \mathbb{C}$ be a non-degenerate homogeneous polynomial of degree $m$. Then we find a basis of $H_{n+2}(\mathbb{C}^{n+2},f^{\leq -\infty},\mathbb{Z})$ by taking $\mu$ Lefshetz thimbles $\{\Gamma^{-}_a,a=1,...,\mu\}$ such that the images of these Lefshetz thimbles under $f$ are the negative real line.
\end{lm}
\begin{proof}
Let $V_s$ be the level set $f^{-1}(s)$. Then for $t>0$, there is canonical map 
$$h_t:V_{-1}\longrightarrow V_{-t}$$
given by
$$h_t(x)=t^{1/m}x$$
Let $\{S_a,a=1,...,\mu\}$ be $\mu$ vanishing cycles on $V_{-1}$, then we can take the Lefshetz thimbles with the form 
$$\Gamma^{-}_a=\{h_t(S_a)\lvert t\in[0,+\infty)\},a=1,...,\mu$$
Easy to see that $f(\Gamma^{-}_a)$ is the negative real line.
\end{proof}
Apply the results above, we can calculate the oscillatory integrals of $\Delta_f$-harmonic forms from oscillatory integrals of holomorphic forms. We have
\begin{thm}\label{harmonic-holomorphic-integral}
Assume that $f:\mathbb{C}^{n+2}\longrightarrow\mathbb{C}$ is a non-degenerate homogeneous polynomial of degree $m$.
Let $\{\alpha_a=i(\alpha_a)+\bar{\partial}_fR_a,a=1,...,\mu\}$ be a basis of $\mathcal{H}$ with polynomial growth holomorhphic parts. Let $\{\Gamma^{-}_b,b=1,...,\mu\}$ be the Lefthetz thimbles in {\bf Lemma} \ref{Lef-thim}. Then we have 
$$\int_{\Gamma^{-}_b}e^{f+\bar{f}}\alpha_a=\int_{\Gamma^{-}_b}e^{2f}i(\alpha_a) \qquad a,b=1,...,\mu$$
and if $\{i(\alpha_a),a=1,...\mu\}$ corresponds to a monomial basis of Milnor ring $R_f$, the oscillatory integrals of holomorphic forms can be calculated as
$$\int_{\Gamma^{-}_b}e^{2f}i(\alpha_a)=\frac{\Gamma(\frac{n+2+deg[i(\alpha_a)]}{m})}{2^{\frac{n+2+deg[i(\alpha_a)]}{m}}}\int_{\Gamma_b\bigcap V_{-1}}\frac{i(\alpha_a)}{df}$$
Here $\frac{i(\alpha_a)}{df}$ is the Gelfand-Leray form (see {\bf Appendix} C), which is a holomorphic $(n+1)$-form on $V_{-1}$.
\end{thm}
\begin{proof}
$f$ takes real values on the Lefshetz thimbles in {\bf Lemma} \ref{Lef-thim}, so $f+\bar{f}$ equals to $2f$ in the integrals, we have
$$\int_{\Gamma^{-}_b}e^{f+\bar{f}}\alpha_a=\int_{\Gamma^{-}_b}e^{2f}i(\alpha_a) \qquad a,b=1,...,\mu$$
The second integral above can be written as
$$\int_{\Gamma^{-}_b}e^{2f}i(\alpha_a)=\int^{+\infty}_0e^{-2t}P(t)dt$$
Here $P(t)$ is 
$$P(t)=\int_{\Gamma_b\bigcap V_{-t}}\frac{i(\alpha_a)}{df}=\frac{1}{2\pi i}\int_{T}\frac{i(\alpha_a)}{f+t}$$
$T$ is the boundary of a small tubular neighborhood of the vanishing cycle $\Gamma_b\bigcap V_{-t}$.
By the homogeneity, we have 
$$P(t)=t^{\frac{n+2-m+deg[i(\alpha_a)]}{m}}P(1)$$
hence
$$\int_{\Gamma^{-}_b}e^{2f}i(\alpha_a)=\frac{\Gamma(\frac{n+2+deg[i(\alpha_a)]}{m})}{2^{\frac{n+2+deg[i(\alpha_a)]}{m}}}\int_{\Gamma_b\bigcap V_{-1}}\frac{i(\alpha_a)}{df}$$
\end{proof}
By above discussion, we have
\begin{thm}\label{close-real-structure}
Assume $f:\mathbb{C}^{n+2}\longrightarrow\mathbb{C}$ is a non-degenerate homogeneous polynomial of degree $n+2$. Then following properties are equivalent\\
(i)The sub-bundle $H^{LG}$ of $\hat{H}^{LG}$ which corresponds to the $n*$ parts in Milnor ring $R_f$ is closed under the real structure $\hat{\kappa}^{LG}$.\\
(ii)The space spanned by the cohomology classes $\{\frac{i(\alpha_a)}{df}\lvert degi(\alpha_a)=d*\}$ forms a real subspace of $H^{n+1}(V_{-1},\mathbb{C})$.
\end{thm}
\begin{proof}
Take the Lefshetz thimbles $\{\Gamma_b^{-},b=1,...,\mu\}$ in {\bf Lemma}\ref{Lef-thim} as a basis of $H_{n+2}(\mathbb{C}^{n+2},f^{\leq -\infty},\mathbb{Z})$.\\
By {\bf Lemma}\ref{iso-harmonic-lef}, $(i)$ is equivalent to the subspace 
$$span_{\mathbb{C}}\{(\int_{\Gamma_1^{-}}e^{f+\bar{f}}\alpha,...,\int_{\Gamma_{\mu}^{-}}e^{f+\bar{f}}\alpha), degi(\alpha)=d*\}$$
is a real subspace of $\mathbb{C}^{\mu}$.\\
By {\bf Theorem} \ref{harmonic-holomorphic-integral}
$$(\int_{\Gamma_1^{-}}e^{f+\bar{f}}\alpha,...,\int_{\Gamma_{\mu}^{-}}e^{f+\bar{f}}\alpha)=c_{degi(\alpha)}(\int_{\Gamma_1^{-}\bigcap V_{-1}}\frac{i(\alpha)}{df},...,\int_{\Gamma_{\mu}^{-}\bigcap V_{-1}}\frac{i(\alpha)}{df})$$
Here $c_{degi(\alpha)}$ is a non-zero real number.
Hence we have 
$$span_{\mathbb{C}}\{(\int_{\Gamma_1^{-}}e^{f+\bar{f}}\alpha,...,\int_{\Gamma_{\mu}^{-}}e^{f+\bar{f}}\alpha), degi(\alpha)=d*\}$$
$$=span_{\mathbb{C}}\{(\int_{\Gamma_1^{-}\bigcap V_{-1}}\frac{i(\alpha)}{df},...,\int_{\Gamma_{\mu}^{-}\bigcap V_{-1}}\frac{i(\alpha)}{df}),degi(\alpha)=d*\}$$
As the vanishing cycles $\{\Gamma_b^{-}\bigcap V_{-1},b=1,...,\mu\}$ forms a basis of $H^{n+1}(V_{-1},\mathbb{Z})$, $(ii)$ is equivalent to 
$$span_{\mathbb{C}}\{(\int_{\Gamma_1^{-}\bigcap V_{-1}}\frac{i(\alpha)}{df},...,\int_{\Gamma_{\mu}^{-}\bigcap V_{-1}}\frac{i(\alpha)}{df}),degi(\alpha)=d*\}$$
forms a real subspace of $H^{n+1}(V_{-1},\mathbb{C})$. So we have proved that $(i)$ and $(ii)$ are equivalent.
\end{proof}
To prove $(ii)$ in the above theorem is true, we need to recall some results about the monodromy of Milnor fibration from \cite{B}. The details will be given in {\bf Appendix} C.\par
From {\bf Theorem} \ref{close-real-structure} $(ii)$, we have
\begin{thm}
Assume $f:\mathbb{C}^{n+2}\longrightarrow\mathbb{C}$ is a non-degenerate homogeneous polynomial of degree $n+2$. Then the cohomology classes $\{\frac{i(\alpha_a)}{df}\lvert degi(\alpha_a)=(n+2)*\}$ are exactly the invariant parts of $H^{n+1}(V_{-1},\mathbb{C})$ under the monodromy of Gauss-Manin connection of Milnor fibration. Moreover, these classes span a real subspace of $H^{n+1}(V_{-1},\mathbb{C})$.
\end{thm}
\begin{proof}
See {\bf Appendix} C for the calculation of monodromy of Gauss-Manin connection.\\
As the Gauss-Manin conncetion is a real connection, the monodromy can be written as a real matrix, and the invariant parts corresponds to the eigensubspace of eigenvalue 1, which is obviously a real subspace.
\end{proof}

As a summary of this subsection, we have the following theorem:

\begin{thm}[The small $tt^*$ geometry structure on Landau-Ginzburg side]
There is a sub-$tt^*$ structure $\E^{\LG}=({H}^{\LG}\to M_{mar}, \kappa^{\LG},\eta^{\LG}, D^{\LG}, C^{\LG}, \bar{C}^{\LG})$ of the big $tt^*$ structure $\widehat{\E}^{\LG}=(\hat{H}^{\LG}\to M, \hat{\kappa}^{\LG},\hat{\eta}^{\LG}, \hat{D}^{\LG}, \hat{C}^{\LG}, \hat{\bar{C}}^{\LG})$, consisting of the following data:\\
(1) The state space $M_{mar}$ is the marginal deformation part in $M$\\
(2) The vector bundle  ${H}^{\LG}\to M_{mar}$ is the subbundle of $\hat{H}^{\LG}\lvert M_{mar}$, which corresponds to the harmonic forms with holomorphic representations in $R^{d*}$.\\
(3) The real structure $\kappa^{\LG}$ is the complex conjugate action on differential forms.\\
(4) The pairing $\eta^{\LG}$, $tt^*$ connections $D^{\LG},\bar{D}^{\LG}$ and Higgs fields $C^{\LG}, \bar{C}^{\LG})$ are the restrictions of the corresponding operators in the big $tt^*$ structure.
\end{thm}

\subsection{$tt^*$ geometry of Calab-Yau model}

Consider the universal deformation of complex structures $\pi:\mathcal{X}\longrightarrow M$ for a Calabi-Yau  manifold $X$. When studying the Weil-Peterson metric on the deformation space $M=H^1(X, T_{X})$ for dim$X=3$, Bershadsky-Cecotti-Ooguri-Vafa \cite{BCOV} proved the special geometry ($tt^*$ geometry) relations, which naturally arise on the variation of Hodge structures. Later, Kontsevich-Barannikov \cite{BK} generalized BCOV's theory to higher-dimensional Calabi-Yau manifolds by constructing a Frobenius manifold structure on the `extended moduli space of deformations' $H(PV_X, \bar{\partial})$. Here we give a higher-dimensional generalization of BCOV's theory by considering only the primitive part of Hodge bundle.
\par
Now consider a non-degenerate homogeneous polynomial
$$f:\mathbb{C}^{n+2}\longrightarrow\mathbb{C}$$
with degree n+2 (the Calabi-Yau condition).\par
Let $X_f$ be the Calabi-Yau $n$-fold determined by $f$ in $\mathbb{CP}^{n+1}$. By Tian-Todorov theorem, the universal deformation space of $X_f$ is (locally) $M=H^1(X, T_X)$. By {\bf Theorem}\ref{intr-thm-modu-numb}, the deformation of complex struture can be given by the marginal deformation of $f$ (except for n=2).\par
Consider the primitive Hodge bundle $\mathcal{H}_{prim}^n$ over $M$, on $t\in M$ the fibre is the space of primitive harmonic n-forms. We have a Hodge filtration by holomorphic subbundles:
$$F^n\mathcal{H}_{prim}^n\subseteq\dots \subseteq F^0\mathcal{H}_{prim}^n=\mathcal{H}_{prim}^n,\quad F^k\mathcal{H}_{prim}^n=\oplus_{p=0}^{n-k}\mathcal{H}_{prim}^{n-p, p}.$$
Take the real form $\kappa$ to be the complex conjugation. By Hodge-Riemann bilinear relation, the pairing
$$g_p(u, v)=i^{2p-n}\int_Xu\wedge\bar{v}$$
is a metric on $\mathcal{H}^{n-p, p}_{prim}$, and we define the $tt^*$ metric on $\mathcal{H}_{prim}^n$ to be $g=\sum_pg_p$. 
\begin{df}\label{CY-conn-higgs}
The $tt^*$ connections $D+\bar{D}$ and the Higgs field $C$, $\bar{C}$ are defined by
$$D(\alpha)=\Pi_{\mathcal{H}_{prim}^{p,n-p}}[\nabla^{GM}(\alpha)], \quad\forall p\in\mathbb{N}, \quad \alpha\in\mathcal{H}_{prim}^{p,n-p},$$
$$\bar{D}(\alpha)=\Pi_{\mathcal{H}_{prim}^{p,n-p}}[\bar{\nabla}^{GM}(\alpha)],\quad\forall p\in\mathbb{N}, \quad \alpha\in\mathcal{H}_{prim}^{p,n-p},$$
$$C(\alpha)=\Pi_{\mathcal{H}_{prim}^{p-1,n-p+1}}[\nabla^{GM}(\alpha)], \quad\forall p\in\mathbb{N}, \quad \alpha\in\mathcal{H}_{prim}^{p,n-p},$$
$$\bar{C}(\alpha)=\Pi_{\mathcal{H}_{prim}^{p+1,n-p-1}}[\bar{\nabla}^{GM}(\alpha)], \quad\forall p\in\mathbb{N}, \quad \alpha\in\mathcal{H}_{prim}^{p,n-p},$$
where the various $\Pi$ are orthogonal projections to certain subspaces.
\end{df} 
The variation of Hodge structures gives a flat Gauss-Manin connection on $\mathcal{H}^n_{prim}$. We now show that it can be decomposed to the $tt^*$ connection and Higgs field.
\begin{thm}We have:
$$\nabla^{GM}=D+C,\quad\bar{\nabla}^{GM}=\bar{D}+\bar{C}.$$
\end{thm}
\begin{proof}Take an arbitrary local section $\alpha$ of $\mathcal{H}_{prim}^{p,n-p}$, we have
\begin{displaymath}
(D+C)(\alpha)=(\Pi_{\mathcal{H}_{prim}^{p,n-p}}+\Pi_{\mathcal{H}_{prim}^{p-1,n-p+1}})(\nabla^{GM}\alpha).
\end{displaymath}
So to prove $\nabla^{GM}=D+C$, by Griffiths transversality, we only need to prove that:
\begin{displaymath}
\nabla^{GM}\alpha\perp F^{p+1}\mathcal{H}^n_{prim}.
\end{displaymath}
First we have
\begin{displaymath}
\int_X\alpha\wedge\overline{F^{p+1}\mathcal{H}^n_{prim}}\equiv0.
\end{displaymath}
Since the Gauss-Manin connection is compatible with the pairing given by the wedge product, we have
\begin{displaymath}
\begin{aligned}
0&=\nabla^{GM}\int_X\alpha\wedge\overline{F^{p+1}\mathcal{H}^n_{prim}}\\
&=\int_X(\nabla^{GM}\alpha)\wedge\overline{F^{p+1}\mathcal{H}^n_{prim}}+\int_X\alpha\wedge(\nabla^{GM}\overline{F^{p+1}\mathcal{H}^n_{prim}})\\
&=\int_X(\nabla^{GM}\alpha)\wedge\overline{F^{p+1}\mathcal{H}^n_{prim}}+\int_X\alpha\wedge\overline{\bar{\nabla}^{GM}F^{p+1}\mathcal{H}^n_{prim}}.
\end{aligned}
\end{displaymath}
Notice that $\bar{\nabla}^{GM}F^{p+1}\mathcal{H}^n_{prim}\subseteq F^{p+1}\mathcal{H}^n_{prim} $, the second part of the last row vanishes and
\begin{displaymath}
\int_X({\nabla}^{GM}\alpha)\wedge\overline{F^{p+1}\mathcal{H}^n_{prim}}\equiv0.
\end{displaymath}
Hence $\nabla^{GM}=D+C$, and similiarly $\bar{\nabla}^{GM}=\bar{D}+\bar{C}$.
\end{proof}
The following $tt^*$ equations follow from above theorem, the flatness of Gauss-Manin connection and that $C$ is holomorphic:
\begin{thm}\textbf{($tt^*$ equations)} The operators $D,\bar{D},C,\bar{C}$ satisfying the following equations: \\
(1) $[C_i,C_j]=[\bar{C}_{\bar {i}},\bar{C}_{\bar {j}}]=0$.\\
(2) $[D_i,\bar{C}_{\bar {j}}]=[\bar{D}_{\bar {i}},C_j]=0$.\\
(3) $[D_i,C_j]=[D_j,C_i]$, $[\bar {D}_{\bar {i}},\bar {C}_{\bar {j}}]=[\bar {D}_{\bar {j}},\bar {C}_{\bar {i}}]$. \\
(4) $[D_i,D_j]=[\bar {D}_{\bar{i}},\bar {D}_{\bar {j}}]=0$, $[D_i,\bar{D}_{\bar{j}}]=-[C_i,\bar{C}_{\bar{j}}]$.\\
\end{thm}
Now we give the $tt*$ geometry structure on Calabi-Yau side, as the summary of the above discussion.
\begin{thm}[$tt^*$ geometry structure on Calabi-Yau side]\label{tt*-CY}
Let $f:\mathbb{C}^{n+2}\longrightarrow \mathbb{C}$ be a non-degenerate polynomial of degree $n+2$. Then we have a $tt^*$ structure
$$\E^{\CY}=(H^{\CY}\to M, \kappa^{\CY},\eta^{\CY}, D^{\CY}, C^{\CY}, \bar{C}^{\CY})$$
with the following data:\\
(1)$M$ is a small neighborhood of $0$ in the deformation space of complex structures of $X_f$.\\
(2)$H^{\CY}$ is the primitive part of the bundle of harmonic $n$-forms on $M$.\\
(3)$\kappa^{\CY}$ is the complex conjugation.\\
(4)$\eta^{\CY}=g(\cdot,\kappa^{\CY}\cdot)$ and $g$ is given by Riemann-Hodge bilinear relation
$$g(u,v)=i^{2p-n}\int_Xu\wedge\bar{v},\qquad u,v\in\mathcal{H}^{p,n-p}_{prim}$$
(5)The $tt^*$ connection and Higgs field are given by {\bf Definition}\ref{CY-conn-higgs}
\end{thm}

\section{Explicit isomorphism of $tt^*$ geometries from LG to CY models}\label{sec3}

\subsection{Isomorphism of State spaces}

Suppose $f\in\mathbb{C}[z_0, \dots, z_{n+1}]$ is a nondegenerate homogeneous polynomial of degree $d=n+2$, then from {\bf Theorem 1.4}, the number of marginal deformations in the universal unfolding of $f$ equals $m=\text{dim}H^1(X_f, T_{X_f})=\dbinom{2n+3}{n+1}-(n+2)^2$(except for $n=2$ and $d=4$). Take a marginal deformation
$$f(z, u)=f(z)+\sum_{i=1}^mu_i\phi_i(z),\quad {\rm deg}(\phi_i)=n+2,\quad i=1, \dots, m.$$
Denote by $M$ the space of deformation parameters $(u^1,\dots, u^m)$, which is a small neighborhood of the origin in $\mathbb{C}^m$. For any fixed parameter $u_0\in M$, $f(z, u_0)$ defines a n-dimensional Calabi-Yau hypersurface $X_{f(z, u_0)}\subset\mathbb{CP}^{n+1}$, so we get a natural fibration $\pi:Y\longrightarrow M$. The fibre $Y_{u_0}$ is just the Calabi-Yau hypersurface  $X_{f(z, u_0)}$. This fibration can also been regard as the deformation of complex structure near the origin 0$\in$M. In this way, $M$ corresponds to a small neighborhood of the origin in $H^1(X_f, T_{X_f})$, and this gives a natural identification between the base space of LG $tt^*$ bundle and the base space of CY $tt^*$ bundle. In the rest of the paper, we will denote the base space by $M$ and the two $tt^*$ geometries by 
$$H^{LG}\longrightarrow M,\quad H^{CY}\longrightarrow M.$$

\subsection{Isomorphism of Frobenius algebras}

\subsubsection*{Residue map and Identification of pairings}\

We first introduce some properties of the residue map in Griffiths and Carlson \cite{CG}. Suppose $f\in\mathbb{C}[z_0, \dots, z_{n+1}]$ is a nondegenerate homogeneous polynomial of degree $d$, then $f$ defines a smooth hypersurface $X_f\subset\mathbb{CP}^{n+1}$. Denote 
$$\Omega=\sum_{i=0}^{n+1}(-1)^iz_idz_0\wedge\dots\hat{dz_i}\dots dz_{n+1}.$$
Then the expression $\Omega_A=\frac{A\Omega}{f^{a+1}}$ defines a rational $(n+1)$-form with $X_f$ as polar locus, provided that the degree of $A$ is chosen to make the quotient homogeneous of degree zero, i.e. deg$A=(a+1)d-n-2$. The rational form $\Omega_A$ is said to have adjoint level $a$.
\par
To each $k$-dimensional cohomology class on the complement of $X_f$, a $(k-1)$-dimensional cohomology class is defined on $X_f$ itself by the topological residue:  Given a $k-1$-cycle $\gamma$ on $X_f$, let $T(\gamma)$ be the $k$-cycle in $\mathbb{CP}^{n+1}-X_f$ defined by forming the boundary of an $\epsilon$-tubular neighborhood of $\gamma$. This construction defines a map
$$T:H_{k-1}(X_f)\longrightarrow H_k(\mathbb{CP}^{n+1}-X_f),$$
whose formal adjoint, up to a factor of $2\pi i$, is the topological residue. We have the explicit relation:
$$\int_{\gamma}{\rm res}\alpha=\frac{1}{2\pi i}\int_{T(\gamma)}\alpha.$$
The following properties of the residue map are proved in \cite{CG}:
\begin{thm}[\cite{CG} Chapter 3]
Let $\Omega^{n+1}(pX_f)$ be the sheaf consisting of meromorphic $(n+1)$-forms with at most poles of order $p$ on $X_f$. Then we have\\
(1) ${\rm res} \Gamma\Omega^{n+1}((n+1)X_f)=H^n_{prim}(X_f,\mathbb{C})$.\\
(2) ${\rm res}\Gamma\Omega^{n+1}((a+1)X_f)=F^{n-a}H^n_{prim}(X_f,\mathbb{C})$.\\
(3) Let $\Omega_A$ be of adjoint level $a$, then ${\rm res}\Omega_{A}$ has Hodge level $n-a+1$ if and only if $A$ lies in the Jacobian ideal.
\end{thm}
The following is a direct corollary of the above properties:
\begin{thm}
The map $A\mapsto ({\rm res}\Omega_A)^{n-a,a}$ induces an isomorphism
$$R_f^{d(a+1)-(n+2)}\longrightarrow H^{n-a,a}_{prim}(X_f),$$
where $({\rm res}\Omega_A)^{n-a,a}$ denotes the $(n-a, a)$-part of ${\rm res}\Omega_A\in H^n_{prim}(X_f,\mathbb{C})$.
\end{thm}
\begin{rem}
In the theorem we denote by $R_f^k$ the degree $k$ part of $R_f$, i.e.
$$R^k_f=\mathbb{C}[z_0, \dots, z_{n+1}]^k/I^k_f,$$
where $\mathbb{C}[z_0, \dots, z_{n+1}]^k$ and $I^k_f$ are the degree $k$ polynomials in $\mathbb{C}[z_0, \dots, z_{n+1}]$ and in $I_f$. Since $I_f$ is a homogeneous ideal, we have $R_f=\oplus_{k=0}^\infty R_f^k$.
\end{rem}
We will also need an explicit C\v{e}ch-theoretic description of $({\rm res}\Omega_A)^{n-a,a}$. Since $X_f$ is smooth, we can cover $\mathbb{CP}^{n+1}$ by $(n+2)$ open subsets:
$$U_j=\left\{\frac{\partial f}{\partial z_j}\neq 0\right\},\quad j=0, \dots, n+1.$$
We have the following result from \cite{CG}:
\begin{lm}[\cite{CG} Proposition 3.1] Let $\Omega_A$ be of adjoint level $a$. Then we have:
$$({\rm res}\Omega_A)^{n-a,a}=c_a\displaystyle\biggl\{\frac{A\Omega_J} {f_J}\displaystyle\biggr\}_{\mid J \mid=a+1},$$
where $f_J=\frac{\partial f}{\partial z_{j_1}}\dots\frac{\partial f}{\partial z_{j_k}}$, $\Omega_J=i\left( \frac{\delta}{\delta z_{j_1}}\wedge\dots\wedge \frac{\delta}{\delta z_{j_k}}\right)\Omega$ for $J=(j_1,\dots, j_k)$, and $c_a=\frac {(-1)^{n+a(a+1)/2}} {a!}$.
\end{lm}
From now on we focus on the case $d=n+2$. Then $X_f$ is a Calabi-Yau hypersurface and $H^{n, 0}(X_f)$ is the (1-dimensional) space of holomorphic volume forms. Denote $\Omega=(-1)^n({\rm res} 1)^{n, 0}$ and consider the isomorphism 
$$\eta: H^a(\wedge^aTX_f)\longrightarrow H^{n-a,a}(X_f),\quad [s]\mapsto [s\vdash\Omega].$$
Note that $\oplus_{a=0}^nH^a(\wedge^aTX_f)$ has a commutative ring structure given by the cup product. We define a ring structure on $H^n(X_f)=\oplus_{a=0}^nH^{n-a,a}(X_f)$ by the Yukawa product
$$H^{n-a, a}(X_f)\times H^{n-b, b}(X_f)\longrightarrow H^{n-a-b, a+b}(X_f),\quad [\alpha]*[\beta]=\eta(\eta^{-1}[\alpha]\wedge\eta^{-1}[\beta]).$$
With respect to this ring structure, we prove that the residue map is a homomorphism:
\begin{thm}
Define $r: \oplus_{a=0}^nR_f^{(n+2)a}\longrightarrow H^n_{prim}(X_f)=\oplus_{a=0}^nH_{prim}^{n-a, a}(X_f)$ by
$$r(A)=({\rm res}\Omega_A)^{n-a, a},\quad\forall {\rm deg}A=(n+2)a$$
and define $r^{\prime}: \oplus_{a=0}^nR_f^{(n+2)a}\longrightarrow H^n_{prim}(X_f)=\oplus_{a=0}^nH_{prim}^{n-a, a}(X_f)$ by 
$$r^{\prime}(A)=c_a^{-1}r(A)$$
then $r^{\prime}$ is a ring isomorphism (up to a sign, which depends on a and n).
\end{thm}
\begin{proof}
It is shown in {\bf Theorem} 3.2 that $r^{\prime}$ is a bijection, so we only need to prove that it is a ring homomorphism. By the definition of product `$*$', this is equivalent to
$$\hat{r}:\oplus_{a=0}^nR_f^{(n+2)a}\longrightarrow\oplus_{a=0}^nH^a(\wedge^aTX_f),\quad [A]\mapsto c_a^{-1}\cdot \eta^{-1}[({\rm res}\Omega_A)^{n-a, a}],\quad \forall{\rm deg}A=(n+2)a$$
being a ring homomorphism. Assume deg$A=(n+2)a$, deg$B=(n+2)b$, we will show that  $\hat{r}(AB)=\hat{r}(A)\wedge\hat{r}(B)$. By {\bf Lemma} 3.4, we have
$$({\rm res}\Omega_A)^{n-a,a}=c_a\displaystyle\biggl\{\frac{A\Omega_I} {f_I}\displaystyle\biggr\}_{\mid I \mid=a+1},$$
$$({\rm res}\Omega_B)^{n-b,b}=c_b\displaystyle\biggl\{\frac{B\Omega_J} {f_J}\displaystyle\biggr\}_{\mid J \mid=b+1},$$
$$({\rm res}\Omega_{AB})^{n-a-b,a+b}=c_{a+b}\displaystyle\biggl\{\frac{AB\Omega_K} {f_K}\displaystyle\biggr\}_{\mid K \mid=a+b+1},$$
$$\Omega=(-1)^n({\rm res} 1)^{n, 0}=\displaystyle\biggl\{\frac {\Omega_i} {f_i}\displaystyle\biggr\}.$$
By taking the contraction with $\Omega$, we have
\begin{displaymath}
\hat{r}(A)=(-1)^a\left\{\frac{A\frac{\delta} {\delta z}_{ I-\{i_0\}}} {f_{ I -\{i_0\}}}\right\}_{I=(i_0,...,i_a)}.
\end{displaymath}
\begin{displaymath}
\hat{r}(B)=(-1)^b\left\{\frac{B\frac{\delta} {\delta z}_{J -\{j_0\}}} {f_{J -\{j_0\}}}\right\}_{J=(j_0,...,j_b)}.
\end{displaymath}
\begin{displaymath}
\hat{r}(AB)=(-1)^{a+b}\left\{\frac{AB\frac{\delta} {\delta z}_{K -\{k_0\}}} {f_{ K -\{k_0\}}}\right\}_{K=(k_0,...,k_{a+b})}.
\end{displaymath}
Here we used the notation
\begin{displaymath}
{\frac{\delta} {\delta z}_{I=(i_0,...,i_a)}}={\frac{\delta} {\delta z_{i_0}}}\wedge...\wedge {\frac {\delta} {\delta z_{i_a}}}.
\end{displaymath}
By direct calculation, we have
\begin{displaymath}
(\hat{r}(AB))_{K=(i_0,...,i_a,j_1,...,j_b)}=(\hat{r}(A))_{ I=(i_0,...,i_a)}\wedge(\hat{r}(B))_{J=(j_0,...,j_b)},\quad for\text{  }i_a=j_0,
\end{displaymath}
hence $\hat{r}(AB)=\hat{r}(A)\wedge\hat{r}(B)$.
\end{proof}
By using {\bf Lemma} 3.4, Griffiths and Carlson \cite{CG} proved the following:
\begin{thm}[\cite{CG} Chapter 3]
$$\int_{X_f}r(A)\cup r(B)=k_{ab}{\rm Res}_f(A,B), \quad degA=(n+2)a,\quad degB=(n+2)b,\quad a+b=n,$$
here $k_{ab}=\frac{(-1)^{\frac{a(a-1)+b(b-1)}{2}+b^2}}{a!b!}$.
\end{thm}
In summary, we have proved the following isomorphism of Frobenius algebras:
\begin{thm}
Let $f\in\mathbb{C}[z_0, \dots, z_{n+1}]$ be a nondegenerate homogeneous polynomial of degree $n+2$. Then the milnor ring $R_f$ is a Frobenius algebra with the multiplication given by $[A]\cdot[B]=[AB]$ and the pairing given by 
$$([A], [B]):={\rm Res}_f(A, B),\quad {\rm deg}A+{\rm deg}B=(n+2)n.$$
Consider the Frobenius algebra structure on $H^n_{prim}(X_f)=\oplus_{a=0}^nH^{n-a, a}(X_f)$, whose multiplication is given by Yukawa product
$$[\alpha]*[\beta]:=\eta(\eta^{-1}[\alpha]\wedge\eta^{-1}[\beta]),$$
where $\eta$ is the isomorphism
$$\eta: H^a(\wedge^aTX_f)\longrightarrow H^{n-a,a}(X_f),\quad [s]\mapsto [s\vdash\Omega],$$
$\Omega=(-1)^n({\rm res} 1)^{n, 0}$ is a holomorphic volume form and $c_a=\frac {(-1)^{n+a(a+1)/2}} {a!}$. And the Frobenius pairing on $H^n_{prim}(X_f)=\oplus_{a=0}^nH^{n-a, a}(X_f)$ is defined by
$$(\alpha, \beta):={\rm Res}_f(A, B),$$
 where
 $$\alpha=({\rm res}\Omega_A)^{n-a, a},\quad \beta=({\rm res}\Omega_B)^{n-b, b},\quad {\rm deg}A=(n+2)a,\quad {\rm deg}B=(n+2)b,\quad a+b=n.$$
Then $\oplus_{a=0}^nR_f^{(n+2)a}$ is a Frobenius subalgebra of $R_f$, and the map $$r^{\prime}: \oplus_{a=0}^nR_f^{(n+2)a}\longrightarrow H^n_{prim}(X_f)=\oplus_{a=0}^nH_{prim}^{n-a, a}(X_f)$$ defined by
$$r^{\prime}(A)=c_a^{-1}({\rm res}\Omega_A)^{n-a, a},\quad\forall {\rm deg}A=(n+2)a$$
is an isomorphism of Frobenius algebras(up to constants, which depend on the gradation and n).
\end{thm}

In a $tt^*$ geometry we have the $tt^*$ metric $g$ and the real form $\kappa$. Define a complex-bilinear form $\eta$ by $\eta(\alpha,\beta)=g(\alpha, \kappa\beta)$. We want to compare $\eta^{LG}$ and $\eta^{CY}$ with their corresponding pairings in Frobenius algebras. Notice that 
$$\eta^{LG}(\alpha, \beta)=\int_{\mathbb{C}^{n+2}}\alpha\wedge*\beta, \quad\forall\Delta_f\text{-harmonic } (n+2)-\text{forms } \alpha, \beta$$
and
$$\eta^{CY}(\alpha, \beta)=i^{a-b}\int_{X_f}\alpha\wedge\beta,\quad\forall \alpha\in H_{prim}^{n-a, a}(X_f),\quad \beta\in H_{prim}^{n-b, b}(X_f),\quad a+b=n.$$

Fan and Shen \cite{FS} proved the following, which corresponds $\eta^{LG}$ to the residue pairing ${\rm Res}_f$ on milnor ring.
\begin{thm}[\cite{FS} {\bf Thorem} 3.4]
Suppose $f\in\mathbb{C}[z_1, \dots, z_{n}]$ is a nondegenerate quasi-homogeneous polynomial, $\alpha, \beta$ are $\Delta_f$-harmonic $n$-forms on $\mathbb{C}^n$. Then there are polynomials $A, B\in\mathbb{C}[z_1, \dots, z_{n}]$ and $(n-1)$-forms $\mu, \nu$ such that $\alpha=Adz_1\wedge\dots\wedge dz_{n}+\bar{\partial}_f\mu$ and $\beta=Bdz_1\wedge\dots\wedge dz_{n}+\bar{\partial}_f\nu$. We have the following:
$$\int_{\mathbb{C}^n}\alpha\wedge*\beta=k_n{\rm Res}_f(A, B),$$
where  $k_n=\frac{(-1)^{n(n-1)/2}i^n}{2^n}$.
\end{thm}
Now we can compare these two Frobenius algebras from LG side and CY side.
Denote by $H^{LG}$ the subbundle of $\hat{H}^{LG}$ defined by $H^{LG}_f\cong\oplus_{a=0}^\infty R_f^{(n+2)a}, \forall f\in M$. Then the bundles $H^{LG}\longrightarrow M$ and $H^{CY}\longrightarrow M$ have isomorphic Frobenius algebra structures on each fiber. If we combine the results of \cite{CG} and \cite{FS}, we have the following theorem:
\begin{thm}
The isomorphism of Frobenius algebras
$$r^{\prime}: H^{LG}_f\cong\oplus_{a=0}^\infty R_f^{(n+2)a}\longrightarrow H^n_{prim}(X_f)=H^{CY}_f,\quad [A]\mapsto c_a^{-1}({\rm res}\Omega_A)^{n-a, a},\quad\forall {\rm deg}A=(n+2)a $$
maps  $\frac{1}{k_{n+2}}\eta^{LG}$ to $\eta^{CY}$(up to a power of $i$, which depends on gradation and $n$).
\end{thm}

\subsection{Identification of Higgs fields}

\subsubsection*{Higgs field on LG side}\

Suppose $f\in\mathbb{C}[z_0, \dots, z_{n+1}]$ is a nondegenerate homogeneous polynomial of degree $d$ and $M$ is the marginal deformation space of $f$. Pick a monomial basis $\phi_1, \dots, \phi_m$ of $R_f^d$, then the marginal deformations 
$$f+u\phi_i,\quad i=1, \dots, m$$
are curves in $M$, whose tangent vectors at $u=0$ constitute a basis of the tangent space $T_fM$, denoted by $\partial_1, \dots, \partial_{n+1}$. We make the identification $T_fM\cong R_f$. In this case, the Higgs field $C: TM\longrightarrow End(H^{LG})$ on LG side is given by the multiplication in milnor ring $R_f$:
\begin{thm}
Denote $C_i=C(\partial_i)\in End(H^{LG}_f)\cong End(R_f)$. Then $C_i[A]=[\phi_iA]$ for $[A]\in R_f$.
\end{thm}
\begin{proof}
Suppose a $\Delta_f$-harmonic form $\alpha\in H^{LG}_f$ corresponds to $[A]\in R_f$, then there is a $(n+1)$-form $\mu$ such that $\alpha=Adz_0\wedge\dots\wedge dz_{n+1}+\bar{\partial}_f\mu$. Recall that $C_i=\Pi\circ\partial_if=\Pi\circ\phi_i$. Thus 
\begin{displaymath}
\begin{aligned}\
C_i\alpha&=\Pi(\phi_i(Adz_0\wedge\dots\wedge dz_{n+1}+\bar{\partial}_f\mu))\\
&=\Pi(\phi_iAdz_0\wedge\dots\wedge dz_{n+1}+\phi_i(\bar{\partial}\mu+\partial f\wedge\mu))\\
&=\Pi(\phi_iAdz_0\wedge\dots\wedge dz_{n+1}+(\bar{\partial}(\phi_i\mu)-\bar{\partial}\phi_i\wedge\mu)+\phi_i\partial f\wedge\mu)\\
&=\Pi(\phi_iAdz_0\wedge\dots\wedge dz_{n+1}+\bar{\partial}(\phi_i\mu)+\partial f\wedge(\phi_i\mu))\\
&=\Pi(\phi_iAdz_0\wedge\dots\wedge dz_{n+1}+\bar{\partial}_f(\phi_i\mu))\\
&=\phi_iAdz_0\wedge\dots\wedge dz_{n+1}+\bar{\partial}_f\nu,
\end{aligned}
\end{displaymath}
which corresponds to $[\phi_iA]\in R_f$ in the identification $H^{LG}_f\cong R_f$.
\end{proof}

\subsubsection*{Higgs field on CY side}\

We identify $T_fM\cong R^{n+2}_f$, and $H^{CY}_f=\ H^n_{prim}(X_f)\cong\oplus_{a=0}^nR_f^{(n+2)a}$ via the isomorphism 
$$r: [A]\mapsto ({\rm res}\Omega_A)^{n-a, a},\quad\forall {\rm deg}A=(n+2)a,\quad a=0, \dots, n.$$
 We prove in the following theorem that the Higgs field on CY side is, up to a constant factor, the multiplication in milnor ring $R_f$:
\begin{thm}
For a polynomial $A$ of degree $(n+2)a$, we have 
$$C_ir(A)=-(a+1)r(\phi_iA).$$
\end{thm}
\begin{proof}
Recall that 
$C=\Pi_{H_{prim}^{n-a-1,a+1}}\circ\nabla^{GM}$ on $H_{prim}^{n-a,a}$. Fix a polynomial deg$A=(n+2)a$, recall that $\Omega_A=\frac{A\Omega}{f^{a+1}}$, which we write more clearly here as $\Omega_{A, f}$. Then
$$\Omega_{A, f+u\phi_i}=\frac{A\Omega}{(f+u\phi_i)^{a+1}},$$
$$\frac{\partial}{\partial u}\Omega_{A, f+u\phi_i}=-(a+1)\frac{\phi_iA\Omega}{(f+u\phi_i)^{a+2}}.$$
Using the definition of Gauss-Manin connection and Griffiths transversality, we compute:
\begin{displaymath}
\begin{aligned}
C_ir(A)&=\Pi_{H_{prim}^{n-a-1,a+1}}(\nabla^{GM}_i({\rm res}\Omega_{A, f})^{n-a, a})\\
&=(\nabla^{GM}_i({\rm res}\Omega_{A, f}))^{n-a-1, a+1}\\
&=(\left.\frac{\partial}{\partial u}\right|_{u=0}({\rm res}\Omega_{A, f+u\phi_i}))^{n-a-1, a+1}\\
&=-(a+1)\left({\rm res}\frac{\phi_iA\Omega}{f^{a+2}}\right)^{n-a-1, a+1}\\
&=-(a+1)r(\phi_iA).
\end{aligned}
\end{displaymath}
\end{proof}
\begin{rem}
Compare the constant a+1 in the above theorem and the constant $c_a$ in the definition of $r^{\prime}$, we have $C_ir^{\prime}(A)=r^{\prime}(\phi_iA)$ up to a sigh depending on $a$ and $n$.
\end{rem}
Combining Theorem 4.1 and 4.2, we have the following

\begin{thm}
Suppose $f\in\mathbb{C}[z_0, \dots, z_{n+1}]$ is a nondegenerate homogeneous polynomial of degree $n+2$,  we have the associated $tt^*$ bundles 
$$H^{LG}\longrightarrow M,\quad H^{CY}\longrightarrow M.$$
Then the Higgs field $C^{LG}$ on $H^{LG}$ is just the multiplication in the subring generated by the marginal part in $R_f$. Via the isomorphism $r^{\prime}:H^{LG}\longrightarrow  H^{CY}$, it differs from $C^{CY}$ by a sign (depending on gradation and n).
\end{thm}

\section{Futher discussion}\label{sec4}

\subsubsection*{Big residue map}\

The residue map we introduced in Section 3 only use a small part of the Milnor ring, one reason is that CY B-model given by the variation of Hodge structure on the deformation is not big enough. In fact it was conjectured that there should be a bigger CY B-model such that we can establish some kind of  LG/CY correspondence (even a $tt^*$ geometry correspondence) between this bigger model and the big LG-B model. The B-model on the extended moduli space has been discussed by S.Barannikov and M.Kontsevich in \cite{BK}, but the geometric meaning of this model is still not clear.\par
On the other hand, given a quasi-homogeneous polynomial $f:\mathbb{C}^n\longrightarrow \mathbb{C}$ with isolated critical point at the origin 0, a bigger residue map can be applied to give a correspondence between natural geometric objects given by the singularity theory of $f:(\mathbb{C}^n,0)\longrightarrow (\mathbb{C},0)$. It was studied by Steenbrink in \cite{Ste}. Chiodo, Iritani and Ruan had used this bigger residue map in their construction of the A-side LG/CY correspondence\cite{CIR}.\par
For simplicity, we shall assume that $f(z_1,...,z_n):\mathbb{C}^n\longrightarrow \mathbb{C}$ is a homogeneous polynomial of degree $d$. By adding a new variable $z_{n+1}$, we can define a family of homogeneous polynomials of degree $d$ by
\begin{displaymath}
F_t(z_1,...,z_n,z_{n+1})=f(z_1,...,z_n)-tz^d_{n+1}
\end{displaymath}
Let $\bar{V}_t$ be the hypersurface determined by $F_t=0$ in the projective space $\mathbb{C}P^n$, $V_t$ be the hypersurface determined by $f=t$ in $\mathbb{C}^{n-1}=\mathbb{C}P^n-\{z_{n+1}=0\}$ and $V_{\infty}$ be the hypersurface determined by $f=0$ in $\mathbb{C}P^{n-1}=\{z_{n+1}=0\}\subseteq \mathbb{C}P^n$. Easy to see that we have $V_{\infty}=\bar{V}_t-V_t$ for any $t\in\mathbb{C}$.\par
Given a complex manifold $X$ and a hypersurface $Y\subseteq X$, the residue map $r:H^m(X-Y)\longrightarrow H^{m-1}(Y))$ can be defined as
\begin{displaymath} 
\int_{\delta (\alpha)}\omega=\int_{\alpha}r(\omega)
\end{displaymath}
In our case, we can define 4 residue maps
\begin{displaymath} 
\begin{aligned}\
&r_1:\mathbb{C}^{n}-V_t\longrightarrow V_t\\
&r_2:V_t\longrightarrow V_{\infty}\\
&r_3:\mathbb{C}^{n}-V_t\longrightarrow \mathbb{C}P^{n-1}- V_{\infty}\\
&r_4:\mathbb{C}P^{n-1}- V_{\infty}\longrightarrow V_{\infty}
\end{aligned}
\end{displaymath}
and 4 Leray coboundary maps $\{\delta_i\}_{i=1,...,4}$ in the similiar way.\par
A simple observation is that for an arbitrary homology class $\beta\in H_{n-2}(V_{\infty})$, we have 
\begin{displaymath}
\delta_1\delta_2(\beta)=\delta_3\delta_4(\beta)
\end{displaymath}
So we have the dual version for the residue map
\begin{displaymath}
r_2r_1(\omega)=r_4r_3(\omega)
\end{displaymath}
Suppose that $\{A_i\}_{i=1}^{\mu}$ is a monomial basis of the Milnor ring $R_f$. We define a function $h$ by
\begin{displaymath}
h(A_i)=\frac{degA_i+n}{d}
\end{displaymath}
In \cite{Ste}, Steenbrink considered the residue of the following rational forms
\begin{displaymath}
\omega_{A_i,t}=(f-t)^{[-h(A_i)]}A_idz_1\wedge...\wedge dz_n
\end{displaymath}
He got the following result in \cite{Ste}
\begin{thm}
Under the residue map $r_1::\mathbb{C}^{n}-V_t\longrightarrow V_t$, $\{\omega_{A_i,t}\}_{i=1}^{\mu}$ is mapped to a basis of the cohomology group $H^{n-1}(V_t)$. Moreover, under the Deligne weight filtration $0=\mathcal{W}_{n-2}\subseteq \mathcal{W}_{n-1}\subseteq \mathcal{W}_{n}=H^{n-1}(V_t)$,
the subset $\{A_i\lvert h(A_i)\notin \mathbb{Z} \}$ is mapped to a basis of $\mathcal{W}_{n-1}$ and the subset $\{A_i\lvert h(A_i)\in \mathbb{Z} \}$ is mapped to a basis of $\mathcal{W}_{n}/\mathcal{W}_{n-1}$ (after projection).
\end{thm}
By simple calculation, we know that the rational forms from the subset $\{A_i\lvert h(A_i)\notin \mathbb{Z} \}$ only has poles along $\bar{V}_t$. These forms are holomorphic along $\mathbb{C}P^{n-1}-V_{\infty}$, so they are mapped to 0 under the residue map $r_3$. As we have $r_2r_1=r_4r_3$, the image of these forms under the residue map $r_1$ is in the kernel of $r_2$. In fact, by the results in chapter 4 of \cite{Ste}, we know that the image is exactly the kernel of $r_2$.\par
The rational forms from $\{A_i\lvert h(A_i)\in \mathbb{Z} \}$ have a simple pole along  $\mathbb{C}P^{n-1}-V_{\infty}$. The residue map in chapter 3 can be regarded as the limit of $r_2r_1=r_4r_3$ as $t\longrightarrow 0$.
\begin{thm}
Given a monomial A with $h(A)\in \mathbb{Z}$. Let $\omega$ be the form determined by $\frac{A\Omega}{f^{[h(A)]}}$ on $\mathbb{C}P^{n-1}$, and $r:\mathbb{C}P^{n-1}-X_f\longrightarrow X_f$. Then we have 
\begin{displaymath}
r(\omega)=lim_{t\longrightarrow 0}r_2r_1(\omega_{A,t})=lim_{t\longrightarrow 0}r_4r_3(\omega_{A,t})
\end{displaymath}
\end{thm}
\begin{proof}
As $r_2r_1=r_4r_3$, we only need to prove that for an arbitrary homology class $\alpha\in H_{n-2}(X_f)$
\begin{displaymath}
\int_{\alpha}r(\omega)=lim_{t\longrightarrow 0}\int_{\alpha}r_4r_3(\omega_{A,t})
\end{displaymath}
Notice that $\omega$ can be got by taking the residue of the form $\frac{Adz_1\wedge...\wedge dz_n}{f^{[h(A)]}}$, i.e.
\begin{displaymath}
\omega=r_3(\frac{Adz_1\wedge...\wedge dz_n}{f^{[h(A)]}})
\end{displaymath}
So we just need to prove that
\begin{displaymath}
\begin{aligned}\
&\int_{\alpha}r(\omega)=\int_{\delta_4(\alpha)}r_3(\frac{Adz_1\wedge...\wedge dz_n}{f^{[h(A)]}})\\
&=lim_{t\longrightarrow 0}\int_{\alpha}r_4r_3(\omega_{A,t})=lim_{t\longrightarrow 0}\int_{\delta_4(\alpha)}r_3(\frac{Adz_1\wedge...\wedge dz_n}{(f-t)^{[h(A)]}})
\end{aligned}
\end{displaymath}
As the residue $r_3(\frac{Adz_1\wedge...\wedge dz_n}{(f-t)^{[h(A)]}})$ converges uniformly to $r_3(\frac{Adz_1\wedge...\wedge dz_n}{f^{[h(A)]}})$ in any compact subset of $\mathbb{C}P^{n-1}-X_f$, and the homology class class $\delta_4(\alpha)$ can be represented by a compact subset in $\mathbb{C}P^{n-1}-X_f$, it is easy to see that the above limit holds.
\end{proof}

\subsubsection*{Remained question: identification of real structures and correspondence for the quasi-homogeneous cases}\
\subsubsection*{Conjecture about real structures}\
Recall that the real form on LG side is just the complex conjugation of $\Delta_f$-harmonic forms, which is well-defined since $\Delta_f$ is a real operator. However, under the identification $H^{LG}_f\cong R_f$, we know little about its action on polynomials. We make the following conjecture that this action corresponds to the complex conjugation on CY side via the isomorphism:
$$r: \oplus_{a=0}^\infty R_f^{(n+2)a}\longrightarrow H^n_{prim}(X_f),\quad [A]\mapsto({\rm res}\Omega_A)^{n-a, a},\quad\forall {\rm deg}A=(n+2)a.$$
\begin{conj}
Suppose $[A], [B]\in R_f$ and $(n-1)$-forms $\mu, \nu$ such that $\alpha=Adz_1\wedge\dots\wedge dz_{n}+\bar{\partial}_f\mu$, $\beta=Bdz_1\wedge\dots\wedge dz_{n}+\bar{\partial}_f\nu$ are $\Delta_f$-harmonic forms. If $\alpha=\bar{\beta}$, then $r(A)=\overline{r(B)}$.
\end{conj}
Since we already proved the correspondence of $\eta(u, v)=g(u, \kappa v)$ and the Higgs field $C$, if the above conjecture is true, we immediately have the following:
\begin{conj}
The map $r^{\prime}: H^{LG}\longrightarrow H^{CY}$  is a `morphism' of $tt^*$ geometries, i.e. it maps the $tt^*$ structures $\E^{LG}=(H^{LG}\longrightarrow M, \kappa^{LG},  g^{LG}, D^{LG}, C^{LG})$ to a constant multiple of $(H^{CY}\longrightarrow M, \kappa^{CY},  g^{CY}, D^{CY}, C^{CY})$ .
\end{conj}
Here is a simple example of the real structure on the LG side, which can be found in \cite{AGV} and \cite{Ce1}.
Let's consider the simplest case $f=x^m$, and we choose the monomial basis of $\Omega^1/df\wedge\Omega^0$ as 
\begin{displaymath}
\omega_k=\lambda_kx^{k-1},k=1,...,m-1
\end{displaymath}
The level set $V_1$ is 
\begin{displaymath}
V_1=\{x_k=e^{\frac{2\pi ik}{m}},k=0,...,m-1\}
\end{displaymath}
A basis of the relative homology group $H_1(\mathbb{C},V_1,\mathbb{Z})$ is given by the segments $l_k$ connecting $x_k$ and $x_{k+1}$. From Example 11.1.3 in \cite{AGV}, we have 
\begin{displaymath}
\int_{l_j}e^{-f}\omega_k=\frac{2i}{m}\lambda_ke^{\frac{2\pi ijk}{m}}sin(\frac{k\pi}{m})\Gamma(\frac{k}{m})
\end{displaymath}
Assume that $\omega_k$ represents a  $\Delta_f$-harmonic form $\alpha$. Let's denote $\omega_k^{\prime]}$ the holomorphic representation of the $\Delta_f$-harmonic $\bar{\alpha}$. We have
\begin{displaymath}
\int_{l_j}e^{-f}\omega_k=\overline{\int_{l_j}e^{-f}\omega_k^{\prime}}
\end{displaymath}
The calculation in \cite{AGV} shows that
\begin{displaymath}
\omega_k^{\prime}=\mu_kX^{m-k-1}
\end{displaymath}
\begin{rem}
The above calculation is a example of oscillatory integrals of holomorphic forms. By using the same discussion in {\bf Theorem} \ref{harmonic-holomorphic-integral}, we see that the above integrals in fact calculate the oscillatory integrals of $\Delta_{f}$-harmonic forms (up to a real constant).
\end{rem}
\subsubsection*{Conjecture about the quasi-homogeneous cases}\
The reason why we only established the correspondence of the homogeneous cases is that when $f$ is just a quasi-homogeneous polynomial, the hypersurface $X_f$ will be in a weighted projective space, which may be an orbifold. As the deformation theory of complex structures and Hodge structures on orbifolds has not been well developped, we can only make a vague conjecture here.
\begin{conj}
There is some kind of deformation theory of complex structure and Hodge structure on the Calabi-Yau orbifold determined by a quasi-homogeneous polynomial $f$, which gives a $tt^*$ structure. Morever, this $tt^*$ will corresponds to the $tt^*$ structure on the monodromy invariant part in Landau-Ginzburg side.
\end{conj}
\begin{rem}
Perhaps such correspondence can be constructed by using the residue map for the qausi-homogeneous case (see \cite{Ste}), just like what we have done in the homogeneous case.
\end{rem}

\appendix

\section{Residue pairing}

In this appendix, we will give the definition and some basic properties of the residue pairing. Also we will give the proof of Theorem 3.8. We follow the expositions in \cite{FS} and use \cite{GH} as a reference.
\par
Suppose a holomorphic function $f$ on $\mathbb{C}^n$ has an isolated critical point $0$. Denote by $B=\left\{z\in\mathbb{C}^n:|z|\leq r\right\}$ a ball with center $0$ such that $0$ is the only critical point of $f$ in $B$. Define $L_i$ to be the divisor of $\partial_if$ and $B_i=B-L_i$. Then $\left\{B_i\right\}_{i=1}^n$ is an open cover of $B^*=B-\left\{0\right\}$. Let $\Gamma$ be the following $n$-chain:
$$\Gamma=\{z: |\partial_if|=\epsilon_i\}$$
with orientation given by $d({\rm arg}\partial_1f)\wedge\dots\wedge d({\rm arg}\partial_nf)\geq 0$. We define the residue of a meromorphic $n$-form
$$\omega=\frac{g(z)dz_1\wedge\dots\wedge dz_n}{\partial_1f(z)\dots\partial_nf(z)},\quad g\in\mathcal{O}(B)$$to be
$${\rm Res}_0\omega=\left( \frac{1}{2\pi i}\right)^n\int_{\Gamma}\omega.$$
Since $\omega$ is holomorphic on $B_1\cap\dots\cap B_n$, it gives a C\v{e}ch cochain in $\check{C}(\underline{B}, \Omega^n)$. Since $\delta\omega=0$, $\omega$ represents a cohomology class in $H^{n-1}(B^*, \Omega^n)$. Denote $\chi_\omega$ to be the image of $\left( \frac{1}{2\pi i}\right)^n\omega$ in the Dolbeault isomorphism: $$H^{n-1}(B^*, \Omega^n)\cong H^{n, n-1}_{\bar{\partial}}(B^*).$$
Since $d=\bar{\partial}$ for $(n, q)$-forms, we have a natural isomorphism 
$H^{n, n-1}_{\bar{\partial}}(B^*)\cong H^{2n-1}_{DR}(B^*)$. We can give another interpretation of ${\rm Res}_0$ as follows:
\begin{lm}
We have the identity
$${\rm Res}_0\omega=\left( \frac{1}{2\pi i}\right)^n\int_{\Gamma}\omega=\int_{S^{2n-1}}\chi_\omega.$$
In addition, if $K_{BM}(z, \zeta)$ is the Bochner-Martinelli kernel:
$$K_{BM}(z,\zeta)=(-1)^n\frac{(n-1)!\sum_{j=1}^{n}(\overline{z_j-\zeta_j})\wedge_{i\neq j}(\overline{dz_i}-\overline{d\zeta_j})}{(2\pi i)^n\vert z-\zeta\vert^{2n}}\wedge d\zeta_i\dots\wedge d\zeta_n$$
and $$F: B\longrightarrow \mathbb{C}^n\times\mathbb{C}^n, \quad z\mapsto (z+f(z), z),$$
then $\chi_\omega=gF^*K_{BM}$.
\end{lm}
It can be proved that for any $g$ in the Jacobi ideal $I=\left \langle\frac{\partial f}{\partial z_1},\dots,\frac{\partial f}{\partial z_1}\right \rangle$, 
$${\rm Res}_0\left(\frac{gdz_1\wedge\dots\wedge dz_n}{\partial_1f\dots\partial_nf}\right)=0.$$
Thus ${\rm Res}_0$ induce the following bilinear pairing:
$${\rm Res}_f: \frac{\mathcal{O}}{I}\times\frac{\mathcal{O}}{I}\longrightarrow\mathbb{C},$$
$${\rm Res}_f(g, h)={\rm Res}_0\left((2\pi i)^n\frac{g(z)h(z)dz_1\wedge\dots\wedge dz_n}{\partial_1f\dots\partial_nf}\right),$$
which is defined to be the residue pairing.
\begin{lm}
The residue pairing ${\rm Res}_f: \frac{\mathcal{O}}{I}\times\frac{\mathcal{O}}{I}\longrightarrow\mathbb{C}$ is non-degenerated, i.e. if a holomorphic function $g$ satisfies ${\rm Res}_f(g,h)=0$ for arbitrary holomorphic function $h$, then $g\in I$.
\end{lm}
Let $f$ be a holomorphic function with isolated critical point $0\in\mathbb{C}^n$. Consider a deformation $F: \mathbb{C}^n\times S\longrightarrow \mathbb{C}$ of $f$ satisfying:\\
(1): For arbitrary $t$ in the deformation space $S$, $f_t$ is a holomorphic function with finite critical points.\\
(2): There exists a ball in $\mathbb{C}^n$ such that all the critical points of $f_t, \forall t\in S$  are in it.\\
Then we can define the residue pairings for $f_t$:
$${\rm Res}_{f_t}(g, h)=\sum_{p_i\in(\nabla f_t)^{-1}(0)}{\rm Res}_{p_i}\left((2\pi i)^n\frac{g(z)h(z)dz_1\wedge\dots\wedge dz_n}{\partial_1f_t\dots\partial_nf_t}\right).$$
We have the following result of continuity:
\begin{lm}
$${\lim_{t \to 0}}{\rm Res}_{f_t}(g, h)={\rm Res}_f(g, h).$$
\end{lm}
Since there is an isomorphism 
$$\frac{\Omega^n}{df_t\wedge\Omega^{n-1}}\cong\frac{\mathcal{O}}{I(f_t)},$$
we can define the residue pairing on $\Omega^n/(df_t\wedge\Omega^{n-1})$ by
$${\rm Res}^t(g(z)dz_1\wedge\dots\wedge dz_n, h(z)dz_1\wedge\dots\wedge dz_n):={\rm Res}_{f_t}(g, h).$$
Now we are ready to prove Theorem 3.8. First we prove the following:
\begin{thm}
Suppose $f$ is a holomorphic tame function with finitely many critical points. Let $\alpha_1=A_1+\bar{\partial}_f\beta_1$, $\alpha_2=A_2+\bar{\partial}_{-f}\beta_2$, where $A_i\in\Omega^n/(df\wedge\Omega^{n-1})$ for $i=1, 2$. If one of $\alpha_1, \alpha_2$ has compact support, then we have the following identity:
$$\int_{\mathbb{C}^n}\alpha_1\wedge\alpha_2={\rm Res}_f(A_1, A_2),$$
where the (global) residue pairing is given by 
$${\rm Res}_f(A_1, A_2)=\sum_{p_i\in(\nabla f)^{-1}(0)}{\rm Res}_{p_i}\left((2\pi i)^n\frac{A_1A_2dz_1\wedge\dots\wedge dz_n}{\partial_1f\dots\partial_nf}\right).$$
\end{thm}
\begin{proof}
Without loss of generality, we assume that $\alpha_1$ has compact support. Denote
$$\tilde{\eta}(\alpha_1, \alpha_2)=\int_{\mathbb{C}^n}\alpha_1\wedge\alpha_2=(-1)^n\eta(\alpha_1, *\alpha_2).$$
First we calculate:\\
\begin{displaymath}
\begin{aligned}
\tilde{\eta}(\alpha_1,\bar{\partial}_{-f}\beta_2)&=(-1)^n\eta(\alpha_1,*\bar{\partial}_{-f}\beta_2)\\
&=(-1)^{n+1}\eta(\alpha_1,\partial_f^{\dag}*\beta_2)\\
&=(-1)^{n+1}\eta(\bar{\partial}_f\alpha_1,*\beta_2)=0.
\end{aligned}
\end{displaymath}
Let $B_{2n}(R)$ be the ball in $\mathbb{C}^n$ with center $0$ and radius $R$, then we have
\begin{displaymath}
\begin{aligned}
\tilde{\eta}(\alpha_1,\alpha_2)&={\lim_{R \to\infty}}\int_{B_{2n}(R)}(A_1+\bar{\partial}_f\beta_1)\wedge A_2\\
&={\lim_{R \to\infty}}\int_{S^{2n-1}(R)}A_2\wedge\beta_1^{0,n-1}
\end{aligned}
\end{displaymath}
We need to calculate $\beta_1$ explicitly. Denote by $L_i$ the divisor defined by $\partial_if$, $U_i=\mathbb{C}^n-L_i$, $U^*=\mathbb{C}^n-\left\{ z: df(z)=0 \right\}=\bigcup U_i$. On the open set $U_i$, we define
$$\lambda_i=\frac{\overline{\partial_if}}{\vert df\vert^2},\quad \rho_i=\lambda_i\partial_if.$$
Then $\sum_i\rho_i=1$ on $U^*$, $\rho_i>0$, i.e. $\left\{ \rho_i \right\}$ is a partition of unity on $U^*$ with respect to the open cover $\{U_i\}$. Set
$$\varpi=dz_1\wedge\dots\wedge dz_n,$$
we wish to solve the equation:
$$\bar{\partial}_f\mu=\bar{\partial}_f\mu+df\wedge\mu=\varpi.$$
Set
$$\gamma_0=\sum_j(-1)^{j-1}\lambda_jdz_1\wedge\dots\wedge\widehat{dz_j}\wedge\dots\wedge dz_n,$$
then we have$$df\wedge\gamma_0=\varpi.$$
Suppose that there are unknown $(n-j-1, j)$-forms $\gamma_j$, such that
$$\mu=\sum_{j=0}^{n-1}(-1)^j\gamma_j.$$
Set $\gamma_n=0$, then the equation $\bar{\partial}_f\mu=\varpi$ becomes 
$$df\wedge\gamma_j=\bar{\partial}\gamma_{j-1},\qquad j=1,\dots,n.$$
Consider the vector field $X=\sum_{i=1}^n\lambda_i\frac{\partial}{\partial z_i}$ on $U^*$. Take the contraction with $X$ on both sides of the above equation, one can observe that if $\iota_X\gamma_j=0$, then $\gamma_j=\iota_X(\bar{\partial}\gamma_{j-1})$. However, if $\gamma_j=\iota_X(\bar{\partial}\gamma_{j-1})$, then
$$\iota_X\gamma_j=\iota_X\circ\iota_X(\bar{\partial}\gamma_{j-1})=0.$$
Thus, the sequence $\{\gamma_i\}_{i=1}^{n-1}$ defined by $\gamma_j=\iota_X(\bar{\partial}\gamma_{j-1})$ is a solution. For $j=1$, we can compute
$$\gamma_1=\sum_{i=1}^n(-1)^{i-1}\lambda_i\sum_{j\neq i}dz_1\wedge\dots\wedge\widehat{dz_i}\wedge\dots\wedge\bar{\partial}\lambda_j\wedge\dots\wedge dz_n.$$
Using induction, 
$$\gamma_k=\sum_{i=1}^{n}(-1)^{i-1}\lambda_i\sum_{j_1\neq i}\cdot\cdot\cdot\sum_{j_k\neq i}dz_1\wedge\cdot\cdot\cdot\wedge\widehat{dz_i}\wedge\cdot\cdot\cdot\wedge\bar{\partial}\lambda_{j_1}\wedge\cdot\cdot\cdot\wedge\bar{\partial}\lambda_{j_k}\wedge\cdot\cdot\cdot\wedge dz_n,$$
where $\bar{\partial}\lambda_{j_s}$ is on the $j_s$-th position. We have given a solution to $\bar{\partial}_f\mu=\varpi$ by $\mu=\sum_{j=0}^{n-1}(-1)^j\gamma_j.$ Note that $\mu^{0, n-1}=(-1)^{n-1}\gamma_{n-1}$. By direct computation,
$$\gamma_{n-1}=\frac{(n-1)!(-1)^{i-1}\bar{\partial}\rho_1\wedge\dots\wedge\widehat{\bar{\partial}\rho_i}\wedge\dots\wedge\bar{\partial}\rho_n}{f_1\dots f_n},$$and
$$\gamma_{n-1}\wedge\varpi=\frac{(n-1)!}{c_n}F^*K_{BM}^{n, n-1},$$
where the holomorphic map $F: \mathbb{C}^n\longrightarrow\mathbb{C}^n\times\mathbb{C}^n$ is given by $F(z)=(z+\nabla f(z), z)$, and
$c_n=(-1)^n\frac{(n-1)!}{(2\pi i)^n}$. Denote $A_1=a_1dz_1\wedge\dots\wedge dz_n$. We pick a smooth cut-off function $\rho$ on $\mathbb{C}^n$, such that $\rho\equiv 0$ on a small neighborhood of $0$ and $\rho\equiv 1$ outside a compact set. Let $\beta_1=-a_1(z)\rho\mu$, then $A_1+\bar{\partial}_f\beta_1=A_1(1-\rho)-a_1\bar{\partial}\rho\mu$ is of compact support, and
$$\beta_1^{0, n-1}=(-1)^na_1(z)\rho\gamma_{n-1}.$$
Thus we have
\begin{displaymath}
\begin{aligned}
\int_{\mathbb{C}^n}\alpha_1\wedge\alpha_2&={\lim_{R \to\infty}}\int_{S^{2n-1}(R)}(-1)^na_2a_1\gamma_{n-1}\wedge\varpi\\
&=\int_{S^{2n-1}(R)}(-1)^na_2a_1\frac{(n-1)!}{c_n}F^*K_{BM}^{n, n-1}\\
&=(2\pi i)^n\int_{S^{2n-1}(\infty)}a_2a_1F^*K_{BM}^{n, n-1}\\
&={\rm Res}(A_1, A_2).
\end{aligned}
\end{displaymath}
\end{proof}
\begin{crl}
Suppose $f$ is a holomorphic tame function with finitely many critical points. Let $\alpha_1=A_1+\bar{\partial}_f\beta_1$, $\alpha_2=A_2+\bar{\partial}_f\beta_2$ be $\Delta_f$-harmonic forms, where $A_i\in\Omega^n/(df\wedge\Omega^{n-1})$ for $i=1, 2$. Then we have the following for some $C_n$ depending only on $n$:
$$\eta(\alpha_1, \alpha_2)=C_n{\rm Res}(A_1, A_2).$$
\end{crl}
\begin{proof}
We have $$*\alpha_2=C_nA_2+*\bar{\partial}_f\beta_2.$$
Since $\alpha_1$ is a $\Delta_f$-harmonic form, using the above theorem, we have
\begin{displaymath}
\begin{aligned}
\eta(\alpha_1, \alpha_2)&=\eta(\alpha_1, A_2+\bar{\partial}_f\beta_2)\\
&=C_n\int_{\mathbb{C}^n}\bar{\partial}_f\beta_1\wedge A_2+\eta(\partial_f^{\dag}\alpha_2, \beta_2)\\
&=C_n{\rm Res}(A_1, A_2).
\end{aligned}
\end{displaymath}
\end{proof}

\section{Proof of the $tt^*$ equations of LG model}

Suppose $f(z)$ is a strongly tame holomorphic function on $\mathbb{C}^n$, $F(z, u)$ is a strong deformation of $f$ with deformation parameter space $M$. Then we have the Hodge bundle $\mathcal{H}\longrightarrow M$ with each fiber $\mathcal{H}_f$ equal to the space of $\Delta_f$-harmonic forms. Recall that in section 2.1, we have defined the following operators:
$$D_i=\Pi\circ\partial_i, \quad\bar{D}_{\bar{i}}=\Pi\circ\bar{\partial}_{\bar{i}}, \quad i=1,\cdot\cdot\cdot,s,$$
$$C_i=\Pi\circ\partial_if, \quad\bar{C}_{\bar{i}}=\Pi\circ\overline{\partial_if}, \quad i=1,\cdot\cdot\cdot,s.$$
In this appendix, we will prove the $tt^*$ equations:
\begin{thm}
The operators $D,\bar{D},C,\bar{C}$ satisfying the following relations: \\
(1) $[C_i,C_j]=[\bar{C}_{\bar{i}},\bar{C}_{\bar{j}}]=0.$\\
(2) $[D_i,\bar{C}_{\bar{j}}]=[\bar{D}_{\bar{i}},C_j]=0$.\\
(3) $[D_i,C_j]=[D_j,C_i], [\bar {D}_{\bar{i}},\bar {C}_{\bar{j}}]=[\bar {D}_{\bar{j}},\bar {C}_{\bar{i}}]$. \\
(4) $[D_i,D_j]=[\bar {D}_{\bar{i}},\bar{D}_{\bar{j}}]=0$, $[D_i,\bar{D}_{\bar{j}}]=-[C_i,\bar{C}_{\bar{j}}]$.\\
\end{thm}
We follow the exposition in \cite{T}. First we cite some commutation formulas from \cite{Fan}:
\begin{lm}
$$[\partial_i, \bar{\partial}_f]=\partial(\partial_if)\wedge,\quad[\partial_i, \bar{\partial}_f^{\dag}]=0,$$
$$[\partial_{\bar{i}}, \bar{\partial}_f]=0,\quad [\partial_{\bar{i}}, \bar{\partial}_f^{\dag}]=\overline{\partial_if_a}(dz^a\wedge)^{\dag},$$
$$[\partial_i, \Delta_f]=[\partial(\partial_if)\wedge, \bar{\partial}_f^{\dag}],\quad [\partial_i, G]=-G(\partial_i\Pi+[\partial_i, \Delta_f]\circ G), $$
$$[\partial_{\bar{i}}, \Delta_f]=[\overline{\partial(\partial_if)\wedge}, \partial_f^{\dag}],\quad[\partial_{\bar{i}}, G]=-G(\partial_{\bar{i}}\Pi+[\partial_{\bar{i}}, \Delta_f]\circ G).$$
\end{lm}
Then we can derive the expressions for $D$ and $C$:
\begin{lm}The following holds on the Hodge bundle $\mathcal{H}_f$.
$$\partial_i=D_i+\partial_f\bar{\partial}_f^{\dag}G(\partial_if),\quad \partial_{\bar{i}}=D_{\bar{i}}+\bar{\partial}_f\partial_f^{\dag}G(\overline{\partial_if}),$$
$$\partial_if=C_i+\bar{\partial}_f\bar{\partial}_f^{\dag}G(\partial_if),\quad \overline{\partial_if}=C_{\bar{i}}+\partial_f\partial_f^{\dag}G(\overline{\partial_if}).$$
\end{lm}
\begin{proof}
We only prove the first one, the rest are similar. From Lemma B.2 and the Hodge identity
$$I=\Pi+G\Delta_f=\Pi+\Delta_f G,$$
we can calculate:
\begin{displaymath}
\begin{aligned}
D_i\alpha&=\Pi\partial_i\alpha=\partial_i\alpha-G\Delta_f\partial_i\alpha=\partial_i\alpha-G[\Delta_f, \partial_i]\alpha\\
&=\partial_i\alpha-G[\partial(\partial_if)\wedge, \bar{\partial}_f^{\dag}]\alpha=\partial_i\alpha+G\bar{\partial}_f^{\dag}(\partial(\partial_if)\wedge\alpha)\\
&=\partial_i\alpha+G\bar{\partial}_f^{\dag}\partial_f((\partial_if)\alpha)=\partial_i\alpha-\partial_f\bar{\partial}_f^{\dag}G((\partial_if)\alpha).
\end{aligned}
\end{displaymath}
Here we have used $\partial(\partial_if)\wedge\alpha=\partial_f((\partial_if)\alpha)$, which is because of $\bar{\partial}_f\alpha=0$.
\end{proof}
We begin the proof of Theorem B.1. For simplicity, we use $\bar{\partial}_f\gamma$ (or $\partial_f\gamma$) to denote any element that we need in im$\bar{\partial}_f$ (or im$\partial_f$). For any $\alpha\in\mathcal{H}^n_f$, to show\\
(1) [$C_i,C_j$]=0:
\begin{displaymath}
\partial_if\partial_jf\alpha=\partial_if(C_j\alpha+\bar{\partial}_f\gamma)=C_iC_j\alpha+\bar{\partial}_f\gamma.
\end{displaymath}
Where we used the relation $[\partial_if,\bar{\partial}_f]$=0. Similarly,
$$\partial_jf\partial_if\alpha=C_jC_i\alpha+\bar{\partial}_f\gamma.$$
Thus $C_iC_j\alpha=C_jC_i\alpha$. The proof of $[\bar{C}_{\bar{i}},\bar{C}_{\bar{j}}]=0$ is similar.\\
(2) [$\bar{D}_{\bar {i}},C_j$]=0: 
\begin{displaymath}
\begin{aligned}
\bar{\partial}_{\bar{i}}C_j\alpha&=\bar{\partial}_{\bar{i}}(\partial_jf\alpha+\bar{\partial}_{f}\gamma)\\
&=\partial_jf\bar{\partial}_{\bar{i}}\alpha+\bar{\partial}_{f}\gamma=\partial_jf(\bar{D}_{\bar{i}}\alpha+\bar{\partial}_{f}\gamma)+\bar{\partial}_{f}\gamma\\
&=\partial_jf\bar{D}_{\bar{i}}\alpha+\bar{\partial}_{f}\gamma=C_j\bar{D}_{\bar{i}}\alpha+\bar{\partial}_f\gamma
\end{aligned}
\end{displaymath}
Here we used $[\partial_{\bar{i}},\bar{\partial}_f]=0$ and $[\partial_if,\bar{\partial}_f]$=0. 
Take the projection on both sides,  we have $\Pi(\bar{\partial}_{\bar{i}}C_j\alpha)=C_j\bar{D}_{\bar{i}}\alpha.$
On the other hand, recall that $\bar{D}_{\bar{i}}=\Pi\circ\bar{\partial}_{\bar{i}}$ ,we have $\bar{D}_{\bar{i}}C_j\alpha=\Pi(\bar{\partial}_{\bar{i}} C_j\alpha)$.
Thus [$\bar{D}_{\bar {i}},C_j$]=0. Similarly $[D_i,\bar{C}_{\bar{j}}]=0$.\\
(3) $[D_i,C_j]=[D_j,C_i]$:
\begin{displaymath}
\begin{aligned}
\partial_iC_j\alpha&=\partial_i(\partial_jf\alpha-\bar{\partial}_f\bar{\partial}^\dag_fG(\partial_jf)\alpha)\\
&=(\partial_i\partial_jf)+\partial_jf\partial_i\alpha+\bar{\partial}_f\gamma-\partial_i(\partial f\wedge)\bar{\partial}^\dag_fG(\partial_jf)\alpha\\
&=(\partial_i\partial_jf)+\partial_jf(D_i\alpha+\partial_f\bar{\partial}^\dag_fG(\partial_if)\alpha)+\bar{\partial}_f\gamma-\partial_i(\partial f\wedge)\bar{\partial}^\dag_fG(\partial_jf)\alpha\\
&=(\partial_i\partial_jf)+\partial_jfD_i\alpha-\partial_j(\partial f\wedge)\bar{\partial}^\dag_fG(\partial_if)\alpha+\partial_f\gamma+\bar{\partial}_f\gamma-\partial_i(\partial f\wedge)\bar{\partial}^\dag_fG(\partial_jf)\alpha
\end{aligned}
\end{displaymath}
So we have
\begin{displaymath}
D_iC_j\alpha=C_jD_i\alpha+\Pi(\partial_i\partial_jf-\partial_j(\partial f\wedge)\bar{\partial}^\dag_fG(\partial_if)\alpha-\partial_i(\partial f\wedge)\bar{\partial}^\dag_fG(\partial_jf)\alpha)
\end{displaymath}

Notice that the index $i$ and $j$ in the second term are symmetric, we have proved that $[D_i,C_j]\alpha=[D_j,C_i]\alpha$. Similarly, $[\bar {D}_{\bar{i}},\bar {C}_{\bar{j}}]=[\bar {D}_{\bar{j}},\bar {C}_{\bar{i}}]$.\\
(4) [$D_i,D_j$]=0:
\begin{displaymath}
\begin{aligned}
\partial_iD_j\alpha&=\partial_i(\partial_j\alpha-\partial_f\bar{\partial}^\dag_fG(\partial_jf)\alpha)\\
&=\partial_i\partial_j\alpha-\partial_f\gamma \qquad([\partial_i,\partial_f]=0.)\\
\end{aligned}
\end{displaymath}
Thus we have
\begin{displaymath}
D_iD_j(\alpha)=\Pi(\partial_i\partial_j\alpha)=D_jD_i(\alpha),
\end{displaymath}
which means [$D_i,D_j$]=0. Similarly $[\bar {D}_{\bar{i}},\bar{D}_{\bar{j}}]=0.$\\
Now the only thing left is to show $[D_i,\bar{D}_{\bar{j}}]=-[C_i,\bar{C}_{\bar{j}}]$.\\
\begin{displaymath}
\begin{aligned}
\partial_i\bar{D}_{\bar{j}}\alpha&=\partial_i(\bar{\partial}_{\bar{j}}\alpha-\bar{\partial}_f\partial^\dag_fG(\overline{\partial_jf}\alpha))\\
&=\partial_i\bar{\partial}_{\bar{j}}\alpha-\bar{\partial}_f\gamma-(\partial_i(\partial f))\partial^\dag_fG(\overline{\partial_jf}\alpha),
\end{aligned}
\end{displaymath}
so we have
\begin{displaymath}
D_i\bar{D}_{\bar{j}}\alpha=\Pi(\partial_i\bar{\partial}_{\bar{j}}\alpha-(\partial_i(\partial f\wedge))\partial^\dag_fG\overline{\partial_jf}\alpha).
\end{displaymath}
Similiarly
\begin{displaymath}
\bar{D}_{\bar{j}}D_i\alpha=\Pi(\bar{\partial}_{\bar{j}}\partial_i\alpha-(\bar{\partial}_{\bar{j}}(\overline{\partial f}\wedge))\bar{\partial}^\dag_fG\partial_if\alpha).
\end{displaymath}
On the other hand, we have
\begin{displaymath}
\begin{aligned}
\partial_if\bar{C}_{\bar{j}}\alpha&=\partial_if(\overline{\partial_jf}\alpha-\partial_f\partial^\dag_fG\overline{\partial_jf}\alpha)\\
&=\partial_if\overline{\partial_jf}\alpha-\partial_f\gamma+(\partial_i(\partial f\wedge))\partial^\dag_fG\overline{\partial_jf}\alpha.
\end{aligned}
\end{displaymath}
Where we have used $[\partial_if, \partial_f]=-\partial(\partial_if)\wedge.$ Thus
\begin{displaymath}
C_i\bar{C}_{\bar{j}}\alpha=\Pi(\partial_if\overline{\partial_jf}\alpha+(\partial_i(\partial f\wedge))\partial^\dag_fG\overline{\partial_jf}\alpha)
\end{displaymath}
and similarly
\begin{displaymath}
\bar{C}_{\bar{j}}C_i\alpha=\Pi(\overline{\partial_jf}\partial_if\alpha+(\bar{\partial}_{\bar{j}}(\overline{\partial f})\wedge)\bar{\partial}^\dag_fG(\partial_if)\alpha).
\end{displaymath}
Now it is obvious that $[D_i,\bar{D}_{\bar{j}}]=-[C_i,\bar{C}_{\bar{j}}]$. We have proved Theorem B.1.
\begin{rem}The $tt^*$ equations are the generalization of the special geometry relations on Calabi-Yau threefolds \cite{St}, which are famous as the genus 0 anomaly equations in \cite{BCOV}. These equations are equivalent to that the connection $\nabla=D+\bar{D}+C_idu^i+\bar{C}_{\bar{i}}du^{\bar{i}}$ on the Hodge bundle is flat. 
\end{rem}

\section{Calculation of monodromy of Gauss-Manin connection}

Assume that $f:\mathbb{C}^{n+2}\longrightarrow\mathbb{C}$ is a non-degenerate quasi-homogenous polynomial.\par
Let's consider the singularity theory of $(f,\mathbb{C}^{n+2},0)$. Let $\Delta$ be a small disc near the origin 0, and $\Delta^*=\Delta-\{0\}$. $f$ gives the Milnor fibration $f^{-1}(\Delta^*)\longrightarrow \Delta^*$. Assocaite each $t\in\Delta^*$ with the cohomology group $H^{n+1}(f_t)$, we get the flat vector bundle $H\longrightarrow \Delta^*$ equipped with the Gauss-Manin conncetion. To calculate the monodromy, we should introduce the Brieskorn lattice. For more details about the Brieskorn lattice, see \cite{B} or \cite{Het2}.\par
We know that the non-singular hypersurfaces $f_t$ are stein manifolds, so every cohomology class in $H^{n+1}(f_t)$ can be represented by a holomorphic (n+1)-form on it.\par
There are two ways to get holomorphic (n+1)-forms on the non-singular hypersurfaces $\{f_t\}_{t\in \Delta^*}$. One way is to restrict a holomorphic (n+1)-form on $\mathbb{C}^{n+2}$ to each $f_t$. This gives a subspace of the space of holomrophic sections of $H\longrightarrow\Delta^*$, we shall denote this subspace by $H_f^{\prime}$. \par
Another way is to take the Gelfand-Leray form of a holomorphic $(n+2)$-form. Given a $(n+2)$-form $\omega$ in $\mathbb{C}^{n+2}$, the Gelfand-Leray form of $\omega$ is a holomorphic form in  $H^{n+1}(f_t)$ defined as follow
\begin{displaymath}
\psi(\omega)=\frac{\omega}{df}
\end{displaymath}
As the non-singular hypersurface $f_t$ is given by a regular value $t\in\mathbb{C}^*$, we msut have $\omega=df\wedge \theta$ in a neighborhood of $f_t$, and we define $\frac{\omega}{df}$ to be $\theta\mid_{f_t}$. The restriction is independent of the choice of the neighborhood and $\theta$, so $\frac{\omega}{df}$ is a well-defined holomorphic (n+1)-form on $f_t$.\par
Taking the Gelfand-Leray forms defines a subspace of the space of holomorphic forms $H\longrightarrow\Delta^*$. We shall denote this subspace by $H_f^{\prime\prime}$.\par 
We have the following result (see \cite{Seb} and \cite{Het2}):
\begin{thm}
$H_f^{\prime}$ and $H_f^{\prime\prime}$ are both free $\mathcal{O}_{\Delta}$-modules of rank $\mu$. Restrict to the germ at 0$\in\Delta$, we have
\begin{displaymath}
H_{f,0}^{\prime}\cong\Omega^{n+1}_{\mathbb{C}^{n+2},0}/(df\wedge\Omega^n_{\mathbb{C}^{n+2},0}+d\Omega^n_{\mathbb{C}^{n+2},0})
\end{displaymath}
\begin{displaymath}
H_{f,0}^{\prime\prime}\cong\Omega^{n+2}_{\mathbb{C}^{n+2},0}/df\wedge d\Omega^n_{\mathbb{C}^{n+2},0}
\end{displaymath}
And there is a natural embedding $H_f^{\prime}\hookrightarrow H_f^{\prime\prime}$, given by
\begin{displaymath}
[\omega]\longrightarrow [df\wedge\omega]
\end{displaymath}
Consider $H_{f,0}^{\prime}$ as a sub-module of $H_{f,0}^{\prime\prime}$, we have
\begin{displaymath}
H_{f,0}^{\prime\prime}/H_{f,0}^{\prime}\cong\Omega^{n+2}_{\mathbb{C}^{n+2},0}/df\wedge\Omega^{n+1}_{\mathbb{C}^{n+2},0}\cong R_f
\end{displaymath}
\end{thm} 
When $f=F$ is a holomorphic function with the only critical point at the origin 0, by the analytic version of Nullstellensatz, we know that there exist a positive integer $\kappa_F$ such that $F^{\kappa_F}\in (\partial_iF)$. For an arbitrary element $[\omega]\in H_{F,0}^{\prime\prime}$, we have $z^{\kappa_F}[\omega]=[F^{\kappa_F}\omega]\in H_{F,0}^{\prime}$. If $F$ is a quasi-homogeneous polynomial, we can take $\kappa_F=1$. In this case $zH_{F,0}^{\prime\prime}\subseteq  H_{F,0}^{\prime}$, but we know from the above theorem that $H_{F,0}^{\prime\prime}$ is a free $\mathcal{O}_{\Delta,0}$ module of rank $\mu$ and $H_{F,0}^{\prime\prime}/H_{F,0}^{\prime}\cong\Omega^{n+2}_{\mathbb{C}^{n+2},0}/dF\wedge\Omega^{n+1}_{\mathbb{C}^{n+2},0}\cong R_F$ has complex dimension $\mu$, so we must have $zH_{F,0}^{\prime\prime}=H_{F,0}^{\prime}$.\par
We have the following theorem about the action of Gauss-Manin connection on a holomorphic section from $H_F^{\prime}$ \cite{B}.
\begin{thm}
Let $F:\mathbb{C}^{n+2}\longrightarrow\mathbb{C}$ be a holomorphic function with the only critical point at the origin 0. Suppose $\omega$ is a holomorphic (n+1)-form near the origin 0$\in\mathbb{C}^{n+2}$, and $[\omega]$ is the holomorphic section of the flat bundle $H\longrightarrow\Delta^*$ by restricting $\omega$ to non-singular hypersurfaces $\{F_t\}_{t\in\Delta^*}$. Then the action of the Gauss-Manin connection on $[\omega]$ has the form
\begin{displaymath}
\nabla_t[\omega]=[\frac{d\omega}{dF}].
\end{displaymath}
\end{thm}
\begin{proof}
Given an arbitrary $t_0\in\Delta^*$. Let $\gamma (t)$ be a flat section of homology classes $\gamma (t) \in H_n(F_t)$ near $t_0$. We need to calculate the following
\begin{displaymath}
\partial_t\int_{\gamma (t)}\omega
\end{displaymath}
By using the residue formula \cite{Le}, we have 
\begin{displaymath}
\int_{\gamma (t)}\omega=\frac{1}{2\pi i}\int_{\delta\gamma (t)}\frac{dF\wedge\omega}{F-t}
\end{displaymath}
Here $\delta:H_{n+1}(F_t)\longrightarrow H_{n+2}(\mathbb{C}^{n+2}-F_t)$ is the Leray coboundary map, which is define by taking the boundary of a tubular neighborhood of the homology class in $H_{n+1}(F_t)$. In our case $\gamma (t)$ is a flat section in a small neighborhood of $t_0\in\Delta^*$, $\delta\gamma (t)$ can be taken independent of $t$. Then we have
\begin{displaymath}
\partial_t\int_{\gamma (t)}\omega=\frac{1}{2\pi i}\partial_t\int_{\delta\gamma (t)}\frac{dF\wedge\omega}{F-t}
\end{displaymath}
\begin{displaymath}
=\frac{1}{2\pi i}\int_{\delta\gamma (t)}\frac{dF\wedge\omega}{(F-t)^2}=\frac{1}{2\pi i}(\int_{\delta\gamma (t)}\frac{d\omega}{F-t}-d\frac{\omega}{F-t})
\end{displaymath}
\begin{displaymath}
=\frac{1}{2\pi i}\int_{\delta\gamma (t)}\frac{d\omega}{F-t}=\frac{1}{2\pi i}\int_{\delta\gamma (t)}\frac{dF\wedge \frac{d\omega}{dF}}{F-t}=\int_{\gamma (t)}\frac{d\omega}{dF}
\end{displaymath}
As we can take an arbitrary flat section $\gamma (t)$ near $t_0$, we have proved that 
\begin{displaymath}
\nabla_t[\omega]=[\frac{d\omega}{dF}].
\end{displaymath}
\end{proof}
Given an element $[\omega]\in H_F^{\prime}$ , we can identify it with $[dF\wedge\omega]\in H_F^{\prime\prime}$. Under this embedding, the action of the Gauss-Manin connection has the form
\begin{displaymath}
\nabla_t[dF\wedge \omega]=[d\omega].
\end{displaymath}
In our case, $F=f$ is a (quasi-)homogeneous polynomial, we have $zH_{f,0}^{\prime\prime}=H_{f,0}^{\prime}$. Consider the germ at 0$\in \Delta$, the action of Gauss-Manin connection on $[\beta]\in H_{f,0}^{\prime\prime}$  can be calculate as 
\begin{displaymath}
\nabla_z[\beta]=\nabla_zz^{-1}(z[\beta])=-\frac{[\beta]}{z}+\frac{\nabla_z(z[\beta])}{z}
\end{displaymath}
The last part $\frac{\nabla_z(z[\beta])}{z}$ can be calculated from the formula of Gauss-Manin connection on $H_{f,0}^{\prime}$.\par
Assume that $f(z_1,...,z_{n+2})$ is a quasi-homogeneous polynomial such that 
\begin{displaymath}
f(\lambda^{w_1}z_1,...,\lambda^{w_{n+2}}z_{n+2})=\lambda^{d}f(z_1,...,z_{n+2})
\end{displaymath}
Then we have 
\begin{displaymath}
f(z_1,...,z_{n+2})=\sum_{i=1}^{n+2}\frac{w_i}{d}z_i\partial_{z_i}f(z_1,...,z_{n+2})
\end{displaymath}
We can define a (n+1)-form $\xi$ as
\begin{displaymath}
\xi=\sum_{i=1}^{n+2}(-1)^{i-1}\frac{w_i}{d}z_i dz_1\wedge...\widehat{dz_i}...\wedge dz_{n+2}
\end{displaymath}
Then we have 
\begin{displaymath}
fdz_1\wedge...\wedge dz_{n+2}=df\wedge \xi
\end{displaymath}
With the above discussion, we can prove the following theorem.
\begin{thm}
Assume that $f(z_1,...,z_{n+2})$ is a quasi-homogeneous polynomial with weights $w=(\frac{w_1}{d},...,\frac{w_{n+2}}{d})$. Take a  monomial $\mathbb{C}$-basis $\{z^{\alpha^i}\}_{i=1}^{\mu}$ of $R_f$. Then under the holomorphic basis $\{[z^{\alpha^i}dz_1\wedge...\wedge dz_{n+2}]\}_{i=1}^{\mu}$ of $H_{f,0}^{\prime\prime}$, we have
\begin{displaymath}
z\nabla_z[z^{\alpha^i}dz_1\wedge...\wedge dz_{n+2}]=(<\alpha^i+1,w>-1)[z^{\alpha^i}dz_1\wedge...\wedge dz_{n+2}]
\end{displaymath}
\end{thm}
\begin{proof}
We have 
\begin{displaymath}
z\nabla_z[z^{\alpha^i}dz_1\wedge...\wedge dz_{n+2}]=-[z^{\alpha^i}dz_1\wedge...\wedge dz_{n+2}]+\nabla_zz[z^{\alpha^i}dz_1\wedge...\wedge dz_{n+2}]
\end{displaymath}
\begin{displaymath}
=-[z^{\alpha^i}dz_1\wedge...\wedge dz_{n+2}]+\nabla_z[fz^{\alpha^i}dz_1\wedge...\wedge dz_{n+2}]
\end{displaymath}
\begin{displaymath}
=-[z^{\alpha^i}dz_1\wedge...\wedge dz_{n+2}]+\nabla_z[z^{\alpha^i}df\wedge \xi]=-[z^{\alpha^i}dz_1\wedge...\wedge dz_{n+2}]+[d(z^{\alpha^i}\xi)]
\end{displaymath}
\begin{displaymath}
=-[z^{\alpha^i}dz_1\wedge...\wedge dz_{n+2}]+[\sum_{j=1}^{n+2}\frac{w_j}{d}\partial_j(z^{\alpha^i}\cdot z_j)dz_1\wedge...\wedge dz_{n+2}]
\end{displaymath}
\begin{displaymath}
=(<\alpha^i+1,w>-1)[z^{\alpha^i}dz_1\wedge...\wedge dz_{n+2}]
\end{displaymath}
\end{proof}
When $f$ is homogeneous polynomial of degree $n+2$, $w=(\frac{1}{n+2},...,\frac{1}{n+2})$, and $<\alpha^i+1,w>-1=<\alpha^i,w>$. So $[z^{\alpha^i}dz_1\wedge...\wedge dz_{n+2}]$ is invariant under the monodromy if and only if deg$z^{\alpha^i}$=$(n+2)k$ for some $k\in\mathbb{Z}_{\geq0}$.\par

\bibliographystyle{plain}

\end{document}